\documentclass[11pt]{article}
\usepackage{amsfonts}
\usepackage{color}
\usepackage{amsmath,amssymb,comment}
\newtheorem{theo}{Theorem}[section]
\newtheorem{lem}[theo]{Lemma}
\newtheorem{cor}[theo]{Corollary}

\newcommand{\mysection}[1]{\section{#1} \setcounter{equation}{0}}
\newcommand{\proof}{{\sc Proof.} \quad}
\newcommand{\proofc}{{\sc Proof} \ }
\newcommand{\be}{\begin{equation} \label}
\newcommand{\ee}{\end{equation}}
\newcommand{\bea}{\begin{eqnarray}\label}
\newcommand{\eea}{\end{eqnarray}}
\newcommand{\bas}{\begin{eqnarray*}}
\newcommand{\eas}{\end{eqnarray*}}
\newcommand{\bit}{\begin{itemize}}
\newcommand{\eit}{\end{itemize}}
\newcommand{\qed}{\hfill$\Box$ \vskip.2cm}
\newcommand{\nn}{\nonumber}
\newcommand{\R}{\mathbb{R}}
\newcommand{\N}{\mathbb{N}}
\newcommand{\pO}{\partial\Omega}

\newcommand{\eps}{\varepsilon}

\newcommand{\hra}{\hookrightarrow}
\newcommand{\io}{\int_\Omega}
\newcommand{\na}{\nabla}
\newcommand{\Del}{\Delta}
\newcommand{\del}{\delta}

\newcommand{\vt}{\vartheta}
\newcommand{\lam}{\lambda}

\newcommand{\sig}{\sigma}
\newcommand{\pa}{\partial}
\newcommand{\bom}{\overline{\Omega}}
\newcommand{\Om}{\Omega}

\newcommand{\ov}{\overline}

\newcommand{\wh}{\widehat}

\newcommand{\hs}{\hspace*}

\newcommand{\vp}{\varphi}
\newcommand{\lbal}{\left\{ \begin{array}{l}}
\newcommand{\lball}{\left\{ \begin{array}{ll}}
\newcommand{\ear}{\end{array} \right.}

\newcommand{\cb}{\color{blue}}

\newcommand{\abs}{\\[5pt]}

\newcommand{\adb}{\allowdisplaybreaks}
\newcommand{\tm}{T_{max}}

\newcommand{\F}{{\mathcal{F}}}

\newcommand{\ovv}{\ov{v}}

\begin{document}
\adb
\title{Global smooth behavior in Kuznetsov and Westervelt type\\
viscous wave equations:\\
A unifying approach covering $W^{1,q}$-small initial data}
\author{
Tahir Boudjeriou\footnote{t.boudjeriou@univ-boumerdes.dz}\\
{\small Institute of Electrical \& Electronic Engineering}\\
{\small University of Boumerdes,  Boumerdes, 35000, Algeria}
\and
Michael Winkler\footnote{michael.winkler@math.uni-paderborn.de}\\
{\small Universit\"at Paderborn, Institut f\"ur Mathematik}\\
{\small 33098 Paderborn, Germany}}
\date{}
\maketitle
\begin{abstract}
\noindent 
In a smoothly bounded domain $\Om\subset\R^n$  with $n\geq 1$ and $a>0$, 
we consider an initial-boundary value problem for the general viscous wave equation
\bas
	h(u,u_t) u_{tt} = \Del u_t + a\Del u + f(u,u_t,\na u,\na u_t)
\eas
which appears in models of nonlinear acoustics wave propagation;
well-established equations of Kuznetsov and Westervelt type form particular examples.\abs
While the existing literature offers extensive results on global solutions for sufficiently small initial data $(u_{0}, u_{0t})=(u, u_{t})|_{t=0}$ in second- and higher-order Sobolev spaces it appears to remain open how far global solvability can be established under smallness conditions involving only first-order Sobolev spaces. The present manuscript addresses this question by proving the existence of global classical solutions together with exponential decay of the pair $(u,u_{t})$ in $ W^{1,r}\times W^{1,p}$-Sobolev spaces whenever the nonlinearities $h$ and $f$ are sufficiently smooth and are such that $h(0,0)>0$ as well as
$f(0,0,0,0)=0$ and $\na f(0,0,0,0)=0$.\abs
\end{abstract}
%
%
%
%
%
\newpage
\section{Introduction}\label{intro}
The propagation of nonlinear waves in acoustic media plays a key role in numerous applications, inter alia in ultrasound technology, thermotherapy, ultrasound cleaning, welding, and sonochemistry (\cite{Cri}, \cite{hamilton_blackstock}, \cite{Kal1}, \cite{Abr}).
In dependence on respectively predominant mechanisms, various approximations of the full compressible Navier-Stokes equations
are used as mathematical descriptions of such processes.
Among the most established scalar models which are accordingly used in nonlinear acoustics are the Kuznetsov equation
\be{K}
	\psi_{tt} = c^2 \Del \psi + d \Del \psi_t + \pa_t \Big( |\na\psi|^2 + \frac{\beta}{2c^2} \psi_t^2 \Big)
\ee
for the velocity potential $\psi$ of an approximately irrotational flow, and the Westervelt equation
\be{W}
	p_{tt} = c^2 \Del p + d\Del p_t + \frac{\beta}{\rho_0 c^2} (p^2)_{tt}
\ee
for the pressure $p=-\rho_0 \psi_t$, where $c$, $d$, $\rho_0$ and $\beta$ denote the small-signal speed of sound, 
the sound diffusivity, the equilibrium density and a dimensionless material parameter, respectively.
For details on aspects of corresponding reductions, as well as on further weakly nonlinear acoustic models, we may refer to
\cite{hamilton_blackstock}, \cite{Kuz}, \cite{westervelt}, \cite{Dek3} and \cite{jordan2016}, for instance.\abs
From a viewpoint of mathematical analysis, rigorously justifying validity of such simplifications of considerably more complex
original models requires, as a fundamental necessary prerequisite, an {\em a posteriori} verification of hypotheses 
on smallness of deviations from equilibrium that have formed the basis for their derivation. 
When posed under homogeneous Dirichlet boundary conditions in a bounded domain $\Om\subset\R^n$ with $n\le 3$, 
for instance, the Westervelt equation (\ref{W})
has the property that if the second-order expression 
\be{01}
	\|p\|_{W^{2,2}(\Om)} + \|p_t\|_{W^{1,2}(\Om)} + \|p_{tt}\|_{L^2(\Om)}
\ee
is suitably small at the initial
time, then a global solution exists which, up to multiplication of constants, inherits this smallness feature and, in fact,
even decays with respect to the norm in $W^{2,2}(\Om)$ in the large time limit (\cite{Kal2});
in \cite{Mey}, it has later on been found that whenever $n\ge 1$
and $p>\max\{\frac{n}{2},\frac{n+4}{4}\}$ is such that $p\ne\frac{3}{2}$, 
initial smallness of 
\be{02}
	\|p\|_{W^{2,p}(\Om)} + \|p_t\|_{W^{2-\frac{2}{p},p}(\Om)}
\ee
is sufficient to similarly ensure global solvability and decay.
For the more strongly nonlinear Kuznetsov equation (\ref{K}), a result from \cite{mizohata_ukai} asserts global existence 
of solutions with trivial large time asymptotics if $n\le 3$ and
\be{03}
	\|\psi\|_{W^{3,2}(\Om)} + \|\psi\|_{W^{2,2}(\Om)}
\ee
is sufficiently small at the initial instant; in fact, the same conclusion has later on been seen to hold even for the 
so-called Blackstock equation in which an additional second-order nonlinearity of the form $K\psi_t \Del\psi$ appears 
on the corresponding right-hand side (\cite{fritz_nikolic_wohlmuth2018}). 
Significant improvements with respect to the differentiability order have been achieved in two more recent works, where
a condition on initial smallness of expressions similar to those in (\ref{02}) has been seen to entail a 
conclusion of this type either when $n\le 3$ and $p=2$ (\cite{nikolic_saidhouari_JEEQ2023}), or $n\ge 1$ 
and $p\in (\max\{1,\frac{n}{2}\},\infty) \setminus \{\frac{3}{2}\}$ are arbitrary (\cite{Mey1}).
Further results of this flavor have been obtained not only in the presence of different boundary conditions and
for associated Cauchy problems in $\R^n$ (\cite{Sim}, \cite{Kal5}, \cite{mizohata_ukai}, \cite{dekkers}),
but also for a number of relatives of (\ref{K}) and (\ref{W}) such as equations of Blackstock-Crighton-Kuznetsov or
Blackstock-Crighton-Westervelt type or systems in which couplings to temperature variables or interaction with microbubbles
are taken into account
(\cite{brunnhuber_kaltenbacher}, \cite{brunnhuber_meyer}, \cite{nikolic_saidhouari_JDE}, \cite{Wil}, \cite{Care}, \cite{Nik5}), 
or even for
Jordan-Moore-Gibson-Thompson type equations involving third-order time derivatives 
(\cite{kaltenbacher_lasiecka_pospieszalska}, \cite{lasiecka_wang}, \cite{lasiecka_wang2}, \cite{nikolic_saidhouari_hereditary}, 
\cite{racke_saidhouari}, \cite{said_houari};
see also \cite{Bon1}, \cite{kaltenbacher_nikolic_M3AS2019} and \cite{kaltenbacher_nikolic_SIMA} for some related studies
on finite timescales).
Further aspects of equations related to (\ref{K}) and (\ref{W}) have been studied in
\cite{Shev}, \cite{Nik-1}, \cite{Dek1}, \cite{Pera}, \cite{Dor}, \cite{chen_arxiv} and \cite{Nik1}, for instance.\abs
{\bf Objectives and main results.} \quad
A common characteristic of the mentioned findings on solvability and large time asymptotics 
is linked to the methods by which they have been achieved: 
The approaches that dominate most of the literature either pursue variational strategies which operate in Hilbert space settings
by esentially relying on unknowns or derivatives thereof as test functions, or alternatively create setups accessible 
to theories of maximal Sobolev regularity in parabolic problems.
In both these scenarios, 
assumptions on smallness of quantities such as those in (\ref{01}) and (\ref{02}), and thus especially on initial smallness 
of some second-order spatial derivatives, are motivated by 
the ambition to expediently control ill-signed nonlinearities 	
by means of appropriate embeddings.
The observation that these nonlinear contributions to (\ref{K}) and (\ref{W}) involve spatial derivatives of
the respective unknowns at most of order one only, however, raises the question whether global solvability in problems
of this form can as well be asserted under smallness assumptions which merely require smallness hypotheses on the initial
data in certain first-order Sobolev spaces.\abs
The present manuscript addresses this in an alternative type of variational framework which aims at the discovery of certain
energy-like properties enjoyed by functionals that involve derivatives exclusively of first order.
Approaches of this form seem to have been pursued in parabolic systems with cross-diffusive ingredients 
(\cite{taowin_subcrit}, \cite{lou_win}), but their use in contexts of problems involving 
wave equations seems limited to few precedents addressing some quite particular and rather far relatives only
(cf., e.g., \cite{claes_win} and \cite{win_AMOP}).\abs
In order to highlight a possible potential for generalization, we will simultaneously examine (\ref{K}) and (\ref{W})
in the framework of the initial-boundary value problem
\be{0}
	\lball
	h(u,u_t) u_{tt} = \Del u_t + a\Del u + f(u,u_t,\na u,\na u_t),
	\qquad & x\in\Om, \ t>0, \\[1mm]
	u=0,            
	\qquad & x\in\pO, \ t>0, \\[1mm]
	u(x,0)=u_0(x), \quad u_t(x,0)=u_{0t}(x)
	\qquad & x\in\Om,
	\ear
\ee
for the unknown function $u$
in a smoothly bounded domain $\Om\subset \R^n$ with $n\ge 1$, where $a>0$ is a fixed parameter, and $h$ as well as $f$ are 
sufficiently regular functions which are such that $h(0,0)>0$, $f(0,0,0,0)=0$ and $\na f(0,0,0,0)=0$, thus ensuring
that along small trajectories, (\ref{0}) indeed can be viewed as a superlinearly forced viscous wave equation.
The PDE in (\ref{0}) reduces to (\ref{K}) on letting
\bas
	& & 
	u:=\psi
	\quad \mbox{and} \quad
	a:=\frac{c^2}{d}
	\qquad \mbox{as well as} \qquad
	\quad h(\xi,\sig):=\frac{1}{d}\Big(1-\frac{\beta}{c^2} \sig\Big) 
	\quad \mbox{and} \quad
	f(\xi,\sig,X,Y):=\frac{2}{d} X\cdot Y \\[2mm]
	& & \hs{70mm}
	\quad \mbox{for } (\xi,\sig,X,Y)\in \R^{2n+2},
\eas
while (\ref{W}) is covered by means of the identifications
\bas
	& &
	u:=p
	\quad \mbox{and} \quad
	a:=\frac{c^2}{d}
	\quad \mbox{as well as} \quad
	\quad h(\xi,\sig):=\frac{1}{d}\Big(1-\frac{2\beta}{\rho_0 c^2} \xi\Big) 
	\quad \mbox{and} \quad
	f(\xi,\sig,X,Y):=\frac{2\beta}{\rho_0 c^2 d} \sig^2 \\[2mm]
	& & \hs{70mm}
	\quad \mbox{for } (\xi,\sig,X,Y)\in \R^{2n+2}.
\eas
In the general setting of (\ref{0}), our main result then asserts that for all initial data which are suitably smooth and satisfy
a smallness condition in $W^{1,r}(\Om)\times W^{1,p}(\Om)$ with suitable $r>1$ and $p>1$, a globally defined smooth solution
exists and undergoes some exponential decay in the large time limit:
\begin{theo}\label{theo18}
  Let $n\ge 1$ and $\Om\subset\R^n$ be a bounded domain with smooth boundary, let $a>0$,
  and suppose that
  \be{h}
	h\in \bigcup_{\iota>0} C^{1+\iota}_{loc} (\R^2)
	\quad \mbox{is such that} \quad
	h(0,0)>0,
  \ee
  that
  \be{f0}
	\lbal
	f \in \bigcup_{\iota>0} C^{1+\iota}_{loc}(\R^{2n+2})
	\qquad \mbox{with} \\[1mm]
	f(0,0,0,0)=0 \mbox{ and } \na f(0,0,0,0)=0
	\ear
  \ee
  and that with some $\lam>1$, for each $M>0$ there exists $K_f(M)>0$ such that
  \be{f1}
	|f(\xi,\sig,X,Y)|
	\le K_f(M) \cdot (1+ |X| + |Y| )^\lam
	\quad \mbox{for all $(\xi,\sig,X,Y)\in \R^{2n+2}$ such that } |\xi|+|\sig|\le M.
  \ee
  Then given any
  \be{18.001}
	p>n\cdot\max\{\lam,2\}
	\quad \mbox{such that} \quad
	p\ge 2\cdot\max\{\lam,2\}
  \ee
  and an arbitrary $r\ge\max\{p,2\lam\}$ fulfilling
  \be{18.01}
	r>2\lam
	\qquad \mbox{if } n=2
  \ee
  and
  \be{18.02}
	\frac{np\lam}{n+p-2} \le r < \frac{np}{n-2}
	\qquad \mbox{if } n\ge 3,
  \ee
  one can find $\eta=\eta(p,r,a,h,f)>0$, $\beta=\beta(p,r,a,h,f)>0$ and $\Gamma=\Gamma(p,r,a,h,f)>0$ with the property that whenever
  \be{init}
	u_0\in C^2(\bom) 
	\mbox{ and } 
	u_{0t} \in \bigcup_{\iota>0} C^{1+\iota}(\bom)
	\quad \mbox{are such that} \quad
	u_0=u_{0t}=0 \mbox{ on } \pO,
  \ee
  and that
  \be{18.1}
	\io |\na u_0|^r \le \eta
	\qquad \mbox{and} \qquad
	\io |\na u_{0t}|^p \le \eta,
  \ee
  the problem (\ref{0}) admits a global classical solution
  \bas
	u\in C^0([0,\infty);C^1(\bom)) \cap C^{2,1}(\bom\times (0,\infty))
  \eas
  with
  \bas
	u_t\in C^0([0,\infty);C^1(\bom)) \cap C^{2,1}(\bom\times (0,\infty))
  \eas
  which is such that $h(u,u_t)\ge\frac{1}{\Gamma}$ in $\bom\times [0,\infty)$, and that
  \be{18.2}
	\|u(\cdot,t)\|_{W^{1,r}(\Om)}
	+ \|u_t(\cdot,t)\|_{W^{1,p}(\Om)}
	\le \Gamma e^{-\beta t}
	\qquad \mbox{for all } t>0.
  \ee
\end{theo}
{\bf Remark.} \quad
We emphasize that the conditions in (\ref{init}) on smoothness of the initial data are imposed here only to ensure that
solutions can be constructed which are classical. 
Since the only quantitative requirement on $u_0$ and $u_{0t}$ here and below will be through (\ref{18.1}) only, 
in the presence of less regular initial data which instead of (\ref{init}) satisfy $u_0\in W_0^{1,r}(\Om)$ 
and $u_{0t}\in W_0^{1,p}(\Om)$ with $r$ and $p$ fulfilling (\ref{18.001})-(\ref{18.02}),
straightforward adaptation of the analysis detailed below can be used to alternatively derive a statement, parallel to 
that from Theorem \ref{theo18}, on global existence of weak solutions.\abs
{\bf Main ideas.} \quad
Starting from a maximally defined local classical solution of (\ref{0}) in $\Om\times (0,\tm)$, 
which together with a rather coarse extensibility criterion (\ref{ext}) will be obtained 
under the assumptions (\ref{h})-(\ref{f1}) and (\ref{init}) via Schauder’s fixed point theorem in Lemma \ref{lem_loc}, 
the first objective will be to refine this blow-up alternative in such a way that
\bas
	\mbox{if $\tm<\infty$, \ then \
	either \ }
	\liminf_{t\nearrow\tm} \inf_{x\in\Om} h\big(u(x,t),u_t(x,t)\big)=0
	\ \mbox{or} \
	\limsup_{t\nearrow\tm} \|u_t(\cdot,t)\|_{W^{1,p}(\Om)}
	= \infty.
\eas
In order to exclude the first possibility herein, we substitute 
\bas
	v:= u_t + au
\eas
and view (\ref{0}) as the evolution system 
\bas
	\lbal
	h(u,u_t)v_t = \Del v + ah(u,u_t)u_t + f(u,u_t,\na u, \na u_t),\\[1mm]
	u_t=v-au,
	\ear
\eas
in which the parabolic character of the first equation will be used to carve out a conditional linear dissipation feature acting
at low level of regularity: Namely, there exists $\del_1=\del_1(h)>0$ such that as long as
\be{05}
	|u| + |u_t| \le \del_1,
\ee
the functional
\be{06}
	\F_1(t):=\frac{1+Ba}{2} \io |\na u|^2
	+ \io H^{(1)}(u,u_t) + B \io H^{(2)}(u,u_t),
\ee
with $p\ge 2$ and suitably chosen $B>0$, and with
\be{07} 
	H^{(1)}(\xi,\sig) := \int_0^\sig \xi h(\xi,\wh{\sig}) d\wh{\sig}
	\quad \mbox{and} \quad
	H^{(2)}(\xi,\sig):=\int_0^\sig \wh{\sig} h(\xi,\wh{\sig})d\wh{\sig},
	\qquad (\xi,\sig)\in\R^2,
\ee
satisfies
\bea{08}
	\frac{d}{dt} \F_1^\frac{p}{2}(t)
	+ \frac{1}{C} \cdot \F_1^\frac{p}{2}(t)
	&\le& C \cdot \bigg\{ \io |u| \cdot |f(u,u_t,\na u, \na u_t)| \bigg\}^\frac{p}{2}
	+ C \cdot \bigg\{ \io |v|\cdot |f(u,u_t,\na u,\na u_t)|\bigg\}^\frac{p}{2} \nn\\
	& & + C \cdot \bigg\{ \io |u|^3 \bigg\}^\frac{p}{2}
	+ C \cdot \bigg\{ \io |v|^3 \bigg\}^\frac{p}{2}
\eea
with some $C=C(p,a,h)>0$ (Corollary \ref{cor14}).
To control the superlinear production terms appearing herein, in the core part of our analysis we will subsequently 
complement this by tracing the evolution of 
\be{09}
	\F(t)
	:= \io |\na v|^p
	+ b_1 \cdot \bigg\{ \io |\na u|^r \bigg\}^\frac{p}{r}
	+ b_2 \F_1^\frac{p}{2}(t)
\ee
with certain $b_i=b_i(p,r,a,h,f)>0$, $i\in\{1,2\}$,
and by hence deriving, yet as long as (\ref{05}) holds, an inequality of the form
\be{010}
	\F'(t) + \frac{1}{C} \F(t)
	+ \Big\{ \frac{1}{C} - C \F^{\kappa_1}(t) \Big\} \cdot \io |\na v|^{p-2} |D^2 v|^2
	\le C \F^{\kappa_2}(t)
\ee
with some $C=C(p,r,a,h,f)>0$ and some $\kappa_1=\kappa_1(p,f)>1$ and $\kappa_2=\kappa_2(p,h,f)>1$
(Section \ref{sect5} and Lemma \ref{lem15}).
Thanks to positivity of $\kappa_1$ and the inequality $\kappa_2>1$, this will imply that if $\F(0)$ is sufficiently small,
then $\F$ actually lies below an exponentially decaying function,
and due to the continuity of the embeddings
from $W^{1,r}(\Om)$ and $W^{1,p}(\Om)$ into $L^\infty(\Om)$, this will enable us to close the loop by showing that
suitably small values of $\F(0)$ ensure not only validity of (\ref{05}) throughout evolution, but also global extensibility
(Section \ref{sect6}).
\mysection{Some functional inequalities}
In this section, we present some functional inequalities that will aid in proving the main result. The first of these is used to take advantage of second-order dissipative effect in Lemma \ref{lem1}.
\begin{lem}\label{lem2}
  Let $p\ge 2$ and $s\ge 1$ be such that $s\le \frac{np}{(n-2)_+}$.
  Then there exists $C(p,s)>0$ such that
  \be{2.1}
	\|\na\vp\|_{L^s(\Om)}^p
	\le C(p,s) \io |\na\vp|^{p-2} |D^2\vp|^2
	\qquad \mbox{for all $\vp\in C^2(\bom)$ fulfilling $\vp=0$ on $\pO$.}
  \ee
\end{lem}
\proof 
  Since the inequality $s\le\frac{np}{(n-2)_+}$ ensures that $\frac{2s}{p} \le \frac{2n}{(n-2)_+}$, the continuity
  of the embedding $W^{1,2}(\Om) \hra L^\frac{2s}{p}(\Om)$ thereby entailed enables us to fix $c_1=c_1(p,s)>0$ such that
  \be{2.2}
	\|\psi\|_{L^\frac{2s}{p}(\Om)}^2 \le c_1 \|\na\psi\|_{L^2(\Om)}^2 + c_1 \|\psi\|_{L^2(\Om)}^2
	\qquad \mbox{for all } \psi\in W^{1,2}(\Om).
  \ee
  Moreover choosing $c_2=c_2(p)>0$ such that in line with a Poincar\'e inequality we have
  \be{2.3}
	\|\psi\|_{L^p(\Om)} \le c_2 \|\na \psi\|_{L^p(\Om)}
	\qquad \mbox{for all } \psi\in W_0^{1,p}(\Om),
  \ee
  given an arbitrary $\vp\in C^2(\bom)$ satisfying $\vp=0$ on $\pO$ we can first apply (\ref{2.2}) to see that
  for each $\eps>0$,
  \bea{2.4}
	\bigg\{ \io (|\na\vp|^2+\eps)^\frac{s}{2} \bigg\}^\frac{p}{s}
	&=& \Big\| (|\na\vp|^2+\eps)^\frac{p}{4} \Big\|_{L^\frac{2s}{p}(\Om)}^2 \nn\\
	&\le& c_1 \Big\| \na (|\na\vp|^2+\eps)^\frac{p}{4} \Big\|_{L^2(\Om)}^2
	+ c_1 \Big\| (|\na\vp|^2+\eps)^\frac{p}{4} \Big\|_{L^2(\Om)}^2 \nn\\
	&\le& \frac{p^2 c_1}{4} \io (|\na\vp|^2+\eps)^\frac{p-4}{2} |\na\vp|^2 |D^2\vp|^2
	+ c_1 \io (|\na\vp|^2+\eps)^\frac{p}{2},
  \eea
  because $\na (|\na\vp|^2+\eps)^\frac{p}{4} = \frac{p}{2} (|\na\vp|^2+\eps)^\frac{p-4}{4} D^2\vp\cdot\na\vp$.
  Here, after an integration by parts using that $|\Del\vp|\le \sqrt{n}|D^2\vp|$ and hence
  \bas
	\na\cdot \big\{ (|\na\vp|^2+\eps)^\frac{p-2}{2} \na\vp\big\}
	&=& (p-2) (|\na\vp|^2+\eps)^\frac{p-4}{2} \na\vp\cdot (D^2\vp\cdot\na\vp)
	+ (|\na\vp|^2+\eps)^\frac{p-2}{2} \Del\vp \\
	&\le& (p-2+\sqrt{n}) (|\na\vp|^2+\eps)^\frac{p-2}{2} |D^2\vp|
  \eas
  shows that in line with the H\"older inequality and (\ref{2.3}),
  \bas
	& & \hs{-16mm}
	\io (|\na\vp|^2+\eps)^\frac{p}{2}
	- \eps \io (|\na\vp|^2+\eps)^\frac{p-2}{2} \\
	&=& \io (|\na\vp|^2+\eps)^\frac{p-2}{2} \na\vp\cdot\na \vp \\
	&=& - \io \vp \na \cdot \big\{ (|\na\vp|^2+\eps)^\frac{p-2}{2} \na\vp\big\} \\
	&\le& (p-2+\sqrt{n}) \io |\vp| (|\na\vp|^2+\eps)^\frac{p-2}{2} |D^2\vp| \\
	&\le& (p-2+\sqrt{n}) \cdot \bigg\{ \io (|\na\vp|^2+\eps)^\frac{p-2}{2} |D^2\vp|^2 \bigg\}^\frac{1}{2}
		\cdot \bigg\{ \io |\vp|^p \bigg\}^\frac{1}{p} \cdot \bigg\{ \io (|\na\vp|^2+\eps)^\frac{p}{2}\bigg\}^\frac{p-2}{2p}
		\\
	&\le& (p-2+\sqrt{n})c_2 \cdot \bigg\{ \io (|\na\vp|^2+\eps)^\frac{p-2}{2} |D^2\vp|^2 \bigg\}^\frac{1}{2}
		\cdot \bigg\{ \io |\na \vp|^p \bigg\}^\frac{1}{p} \cdot \bigg\{ \io (|\na\vp|^2+\eps)^\frac{p}{2}\bigg\}^\frac{p-2}{2p}
		\\
	&\le& (p-2+\sqrt{n})c_2 \cdot \bigg\{ \io (|\na\vp|^2+\eps)^\frac{p-2}{2} |D^2\vp|^2 \bigg\}^\frac{1}{2} 
		\cdot \bigg\{ \io (|\na\vp|^2+\eps)^\frac{p}{2}\bigg\}^\frac{1}{2}.
  \eas
  Therefore,
  \bas
	\io (|\na\vp|^2+\eps)^\frac{p}{2}
	\le (p-2+\sqrt{n})^2 c_2^2 \io (|\na\vp|^2+\eps)^\frac{p-2}{2} |D^2\vp|^2
	+ 2 \eps \io (|\na\vp|^2+\eps)^\frac{p-2}{2},
  \eas
  By again estimating $|\na\vp|^2 \le |\na\vp|^2+\eps$ on the right-hand side therein, from (\ref{2.4}) we thus obtain that
  \bas
	\bigg\{ \io (|\na\vp|^2+\eps)^\frac{s}{2} \bigg\}^\frac{p}{s}
	&\le& \Big\{ \frac{p^2 c_1}{4} + (p-2+\sqrt{n})^2 c_1 c_2^2\Big\} \io (|\na\vp|^2+\eps)^\frac{p-2}{2} |D^2\vp|^2 \\
	& & + 2 c_1 \eps \io (|\na\vp|^2+\eps)^\frac{p-2}{2},
  \eas
  from which upon letting $\eps\searrow 0$ we obtain (\ref{2.1}) with $C(p,s):=\frac{p^2 c_1}{4} + (p-2+\sqrt{n})^2 c_1 c_2^2$,
  because $p\ge 2$.
\qed
The second inequality interpolates gradients between weighted $L^2$-norms of the Hessians and certain first-order Lebesgue norms. It is obtained by combining the first inequality with a standard Gagliardo-Nirenberg type estimate, taking into account the same regularization of $|\na \vp|^2$ as $|\na \vp|^2 + \eps$ used in Lemma \ref{lem2}.
\begin{lem}\label{lem3}
  Let $p\ge 2$ and $q\in [p,\frac{(n+2)p}{n}]$. Then there exists $C(p,q)>0$ such that
  \be{3.1}
	\io |\na\vp|^q
	\le C(p,q) \cdot \bigg\{ \io |\na\vp|^p \bigg\}^\frac{q-p}{p} \cdot \io |\na\vp|^{p-2} |D^2 \vp|^2
	\qquad \mbox{for all $\vp\in C^2(\bom)$ fulfilling $\vp=0$ on $\pO$.}
  \ee
\end{lem}
\proof 
  As our assumptions particularly ensure that $\frac{2q}{p} \le \frac{2n}{(n-2)_+}$, we may employ a Gagliardo-Nirenberg inequality
  to fix $c_1=c_1(p,q)>0$ such that
  \be{3.2}
	\|\psi\|_{L^\frac{2q}{p}(\Om)}^\frac{2q}{p}
	\le c_1 \|\na\psi\|_{L^2(\Om)}^\frac{n(q-p)}{p} \|\psi\|_{L^2(\Om)}^\frac{np-(n-2)q}{p}
	+ c_1 \|\psi\|_{L^2(\Om)}^\frac{2q}{p}
	\qquad \mbox{for all } \psi\in W^{1,2}(\Om),
  \ee
  while invoking Lemma \ref{lem2} we find $c_2=c_2(p)>0$ satisfying
  \be{3.3}
	\io |\na\psi|^p \le c_2 \io |\na\psi|^{p-2} |D^2\psi|^2
	\qquad \mbox{for all $\psi\in C^2(\bom)$ such that $\psi=0$ on $\pO$.}
  \ee
  Now for fixed $\vp\in C^2(\bom)$ fulfilling $\vp=0$ on $\pO$, we abbreviate $I(\vp):=\io |\na\vp|^{p-2} |D^2\vp|^2$
  and note that then $\io \Big| \na |\na\vp|^\frac{p}{2}\Big|^2 \le \frac{p^2}{4} I(\vp)$, so that from (\ref{3.2}) and (\ref{3.3})
  we obtain that since $(n+2)p-nq \ge 0$,
  \bas
	\io |\na\vp|^q
	&=& \Big\| |\na\vp|^\frac{p}{2}\big\|_{L^\frac{2q}{p}(\Om)}^\frac{2q}{p} \\
	&\le& c_1 \Big\| \na |\na\vp|^\frac{p}{2}\Big\|_{L^2(\Om)}^\frac{n(q-p)}{p}
		\Big\| |\na\vp|^\frac{p}{2}\Big\|_{L^2(\Om)}^\frac{np-(n-2)q}{p} 
	+ c_1 \Big\| |\na\vp|^\frac{p}{2}\Big\|_{L^2(\Om)}^\frac{2q}{p} \\
	&\le& c_1 I^\frac{n(q-p)}{2p}(\vp) \cdot \bigg\{ \io |\na\vp|^p \bigg\}^\frac{np-(n-2)q}{2p}
	+ c_1 \cdot \bigg\{ \io |\na\vp|^p \bigg\}^\frac{q}{p} \\
	&=& c_1 I^\frac{n(q-p)}{2p}(\vp) \cdot \bigg\{ \io |\na\vp|^p \bigg\}^\frac{q-p}{p}
		\cdot \bigg\{ \io |\na\vp|^p \bigg\}^\frac{(n+2)p-nq}{2p}
	+ c_1 \cdot \bigg\{ \io |\na\vp|^p \bigg\}^\frac{q-p}{p} \cdot \io |\na\vp|^p \\
	&\le& c_1 I^\frac{n(q-p)}{2p}(\vp) \cdot \bigg\{ \io |\na\vp|^p \bigg\}^\frac{q-p}{p}
		\cdot \big(c_2 I(\vp)\big)^\frac{(n+2)p-nq}{2p}
	+ c_1 \cdot \bigg\{ \io |\na\vp|^p \bigg\}^\frac{q-p}{p} \cdot c_2 I(\vp).
  \eas
  Since $\frac{n(q-p)}{2p} + \frac{(n+2)p-nq}{2p}=1$, this yields (\ref{3.1}) with $C(p,q):=c_1 c_2^\frac{(n+2)p-nq}{2p} + c_1 c_2$.
\qed
The third inequality interpolates gradients between weighted $L^2$-norms of Hessians and $L^2$ norms. Its consequence, stated in Corollary \ref{cor312}, will be used in the proof of Lemma \ref{lem16}.
\begin{lem}\label{lem311}
  Let $p\ge 2$ and $r\ge p$ satisfy $r\le\frac{np}{(n-2)_+}$.
  Then there exists $C(p,r)>0$ such that for each $\vp\in C^2(\bom)$ fulfilling $\vp=0$ on $\pO$,
  \bea{311.1}
	\|\na\vp\|_{L^r(\Om)}^p
	&\le& C(p,r)\cdot\bigg\{ \io |\na\vp|^{p-2} |D^2\vp|^2 \bigg\}^\frac{1+\frac{n}{2}-\frac{n}{r}}{1+\frac{n}{2}-\frac{n-2}{p}}
		\cdot \bigg\{ \io \vp^2 \bigg\}^\frac{1-\frac{n}{2}+\frac{np}{2r}}{1+\frac{n}{2}-\frac{n-2}{p}} \nn\\
	& & + C(p,r)\cdot\bigg\{ \io |\na\vp|^{p-2} |D^2\vp|^2 \bigg\}^\frac{1+\frac{n}{2}-\frac{n}{p}}{1+\frac{n}{2}-\frac{n-2}{p}}
		\cdot \bigg\{ \io \vp^2 \bigg\}^\frac{1}{1+\frac{n}{2}-\frac{n-2}{p}}.
  \eea
\end{lem}
\proof
  We first consider the case $r=p$, in which for fixed $\vp$ as indicated, we integrate by parts and use the fact that 
  $|\Del\vp|\le \sqrt{n}|D^2\vp|$ along with the Cauchy-Schwarz inequality to see on abbreviating 
  $I(\vp):=\io |\na\vp|^{p-2} |D^2\vp|^2$ that
  \bea{311.02}
	\|\na\vp\|_{L^p(\Om)}^p
	&=& \io |\na\vp|^{p-2} \na\vp\cdot\na\vp \nn\\
	&=& - \io \vp\cdot \Big\{ (p-2) |\na\vp|^{p-4} \na\vp\cdot (D^2\vp\cdot\na\vp) + |\na\vp|^{p-2} \Del\vp\Big\} \nn\\
	&\le& (p-2+\sqrt{n}) \io |\vp| \cdot |\na\vp|^{p-2} |D^2\vp| \nn\\
	&\le& (p-2+\sqrt{n}) I^\frac{1}{2}(\vp) \cdot \bigg\{ \io \vp^2 |\na\vp|^{p-2}\bigg\}^\frac{1}{2}.
  \eea
  Here, by the H\"older inequality,
  \be{311.2}
	\bigg\{ \io \vp^2 |\na\vp|^{p-2}\bigg\}^\frac{1}{2}
	\le \bigg\{ \io |\vp|^p \bigg\}^\frac{1}{p} \cdot \bigg\{ \io |\na\vp|^p \bigg\}^\frac{p-2}{2p},
  \ee
  whence taking $c_1=c_1(p)>0$ such that in line with a Gagliardo-Nirenberg inequality we have
  \bas
	\|\psi\|_{L^p(\Om)}
	\le c_1 \|\na\psi\|_{L^p(\Om)}^\frac{\frac{n}{2}-\frac{n}{p}}{1+\frac{n}{2}-\frac{n}{p}} 
		\|\psi\|_{L^2(\Om)}^\frac{1}{1+\frac{n}{2}-\frac{n}{p}}
	\qquad \mbox{for all } \psi\in W_0^{1,p}(\Om),
  \eas
  we all in all infer from (\ref{311.02}) and (\ref{311.2}) that, with $J(\vp):=\|\vp\|_{L^2(\Om)}^2$,
  \bas
	\|\na\vp\|_{L^p(\Om)}^p
	&\le& (p-2+\sqrt{n}) \cdot \Big\{ c_1 I^\frac{1}{2}(\vp) 
		\|\na\vp\|_{L^p(\Om)}^\frac{\frac{n}{2}-\frac{n}{p}}{1+\frac{n}{2}-\frac{n}{p}}
		\cdot J^{\frac{1}{2}\cdot\frac{1}{1+\frac{n}{2}-\frac{n}{p}}}(\vp) \Big\} \cdot \|\na\vp\|_{L^p(\Om)}^\frac{p-2}{2} \\
	&=& c_1\cdot (p-2+\sqrt{n}) I^\frac{1}{2}(\vp) J^{\frac{1}{2}\cdot \frac{1}{1+\frac{n}{2}-\frac{n}{p}}}(\vp)
		\cdot \|\na\vp\|_{L^p(\Om)}^\frac{\frac{p}{2}+\frac{np}{4}-1-\frac{n}{2}}{1+\frac{n}{2}-\frac{n}{p}},
  \eas
  that is,
  \bas
	\|\na\vp\|_{L^p(\Om)}^{p-\frac{\frac{p}{2}+\frac{np}{4}-1-\frac{n}{2}}{1+\frac{n}{2}-\frac{n}{p}}}
	\le c_1\cdot (p-2+\sqrt{n}) I^\frac{1}{2}(\vp) J^{\frac{1}{2}\cdot \frac{1}{1+\frac{n}{2}-\frac{n}{p}}}(\vp).
  \eas
  Since
  \bas
	p-\frac{\frac{p}{2}+\frac{np}{4}-1-\frac{n}{2}}{1+\frac{n}{2}-\frac{n}{p}}
	= \frac{p+\frac{np}{2}-n-\frac{p}{2}-\frac{np}{4}+1+\frac{n}{2}}{1+\frac{n}{2}-\frac{n}{p}}
	= \frac{1}{2} \cdot \frac{(1+\frac{n}{2})p -n+2}{1+\frac{n}{2}-\frac{n}{p}} 
	> 0,
  \eas
  this is equivalent to the inequality
  \be{311.4}
	\|\na\vp\|_{L^p(\Om)}^p
	\le c_2 I^\frac{1+\frac{n}{2}-\frac{n}{p}}{1+\frac{n}{2}-\frac{n-2}{p}}(\vp)
		J^\frac{1}{1+\frac{n}{2}-\frac{n-2}{p}}(\vp)
  \ee
  with $c_2\equiv c_2(p):=\big\{ c_1\cdot (p-2+\sqrt{n})\big\}^{2(1+\frac{n}{2}-\frac{n}{p})/(1+\frac{n}{2}-\frac{n-2}{p})}$,
  and thereby establishes (\ref{311.1}) in this special case.\abs
  If $r>p$ is such that $r\le\frac{np}{(n-2)_+}$, however, the corresponding version of (\ref{311.1}) results from (\ref{311.4})  
  upon another application of the Gagliardo-Nirenberg inequality:
  Indeed, according to the latter it follows from these restrictions on $r$ that we can find $c_3=c_3(p,r)>0$ fulfilling
  \bas
	\|\psi\|_{L^\frac{2r}{p}(\Om)}^2
	\le c_3 \|\na\psi\|_{L^2(\Om)}^\frac{n(r-p)}{r} \|\psi\|_{L^2(\Om)}^\frac{2r-nr+np}{r}
	+ c_3 \|\psi\|_{L^2(\Om)}^2
	\qquad \mbox{for all } \psi\in W^{1,2}(\Om),
  \eas
  which in conjunction with (\ref{311.4}) and the fact that
  \bas
	\Big\| \na |\na\vp|^\frac{p}{2}\Big\|_{L^2(\Om)}^2
	&=& \io \Big| \frac{p}{2} |\na\vp|^\frac{p-4}{2} D^2 \vp\cdot\na\vp\Big|^2 \\
	&\le& \frac{p^2}{4} \io |\na\vp|^{p-2} |D^2\vp|^2
	= \frac{p^2}{4} I(\vp)
  \eas
  shows that
  \bas
	\|\na\vp\|_{L^r(\Om)}^p
	&=& \Big\| |\na\vp|^\frac{p}{2}\Big\|_{L^\frac{2r}{p}(\Om)}^2 \\
	&\le& c_3 \Big\| \na |\na\vp|^\frac{p}{2}\Big\|_{L^2(\Om)}^\frac{n(r-p)}{r} 
		\Big\| |\na\vp|^\frac{p}{2}\Big\|_{L^2(\Om)}^\frac{2r-nr+np}{r}
		+ c_3 \Big\| |\na\vp|^\frac{p}{2}\Big\|_{L^2(\Om)}^2 \\
	&\le& c_3\cdot\Big(\frac{p^2}{4}\Big)^\frac{n(r-p)}{2r} I^\frac{n(r-p)}{2r}(\vp) 
		\|\na\vp\|_{L^p(\Om)}^\frac{p(2r-nr+np)}{2r}
		+ c_3 \|\na\vp\|_{L^p(\Om)}^p \\
	&\le& c_3\cdot\Big(\frac{p^2}{4}\Big)^\frac{n(r-p)}{r} I^\frac{n(r-p)}{2r}(\vp) \cdot
		\bigg\{ c_2 I^\frac{1+\frac{n}{2}-\frac{n}{p}}{1+\frac{n}{2}-\frac{n-2}{p}}(\vp)
		J^\frac{1}{1+\frac{n}{2}-\frac{n-2}{p}}(\vp)\bigg\}^\frac{2r-nr+np}{2r}  \\
	& & + c_3 c_2 I^\frac{1+\frac{n}{2}-\frac{n}{p}}{1+\frac{n}{2}-\frac{n-2}{p}}(\vp) 
		J^\frac{1}{1+\frac{n}{2}-\frac{n-2}{p}}(\vp).
  \eas
  Since a straightforward computation confirms that
  \bas
	\frac{n(r-p)}{2r} 
	+ \frac{1+\frac{n}{2}-\frac{n}{p}}{1+\frac{n}{2}-\frac{n-2}{p}} \cdot \frac{2r-nr+np}{2r}
	= \frac{1+\frac{n}{2}-\frac{n}{r}}{1+\frac{n}{2}-\frac{n-2}{p}},
  \eas
  and since clearly
  \bas
	\frac{1}{1+\frac{n}{2}-\frac{n-2}{p}} \cdot \frac{2r-nr+np}{2r} 
	= \frac{1-\frac{n}{2}+\frac{np}{2r}}{1+\frac{n}{2}-\frac{n-2}{p}},
  \eas
  this completes the proof.
\qed
\begin{cor}\label{cor312}
  Let $p\ge 2$ and $r\ge p$ be such that $r<\frac{np}{(n-2)_+}$.
  Then for each $\eps>0$ there exists $\Gamma_1(\eps,p,r)>0$ with the property that whenever $\vp\in C^2(\bom)$
  satisfies $\vp=0$ on $\pO$,
  \be{312.1}
	\|\na\vp\|_{L^r(\Om)}^p
	\le \eps \io |\na\vp|^{p-2} |D^2\vp|^2
	+ \Gamma_1(\eps,p,r) \cdot \bigg\{ \io \vp^2 \bigg\}^\frac{p}{2}.
  \ee
\end{cor}
\proof
  Since the strict inequality $r<\frac{np}{(n-2)_+}$ ensures that $(1+\frac{n}{2}-\frac{n}{r})/(1+\frac{n}{2}-\frac{n-2}{p}) <1$,
  and since furthermore clearly also $(1+\frac{n}{2}-\frac{n}{p})/(1+\frac{n}{2}-\frac{n-2}{p}) <1$,
  this can readily be concluded upon applying Young's inequality to (\ref{311.1}).
\qed

\mysection{Local existence and extensibility}
The purpose of this section is to derive a basic statement on local existence of classical solutions to (\ref{0}),
as well as a handy criterion for their global extensibility.\abs
Our first step into this direction applies standard parabolic theory in a suitable self-map framework to
obtain the following by means of the Schauder fixed point theorem:
\begin{lem}\label{lem_loc}
  Suppose that $h\in \bigcup_{\iota>0} C^{1+\iota}_{loc} (\R^2)$ and
  $f\in \bigcup_{\iota>0} C^{1+\iota}_{loc} (\R^{2n+2})$, and that $u_0$ as well as $u_{0t}$ satisfy (\ref{init})
  and are such that
  \bas
	h(u_0(x),u_{0t}(x))>0
	\qquad \mbox{for all } x\in\bom.
  \eas
  Then there exist $\tm\in (0,\infty]$ as well as a function
  \bas
	u \in C^0([0,\tm);C^1(\bom)) \cap C^{2,1}(\bom\times (0,\tm))
  \eas
  such that
  \bas
	u_t \in C^0([0,\tm);C^1(\bom)) \cap C^{2,1}(\bom\times (0,\tm)),
  \eas
  that 
  \bas
	h\big(u(x,t),u_t(x,t)\big)>0
	\qquad \mbox{for all $x\in\bom$ and } t\in [0,\tm),
  \eas
  that (\ref{0}) is satsified in the classical sense in $\Om\times (0,\tm)$, and that
  \bea{ext}
	& & \hs{-20mm}
	\mbox{if $\tm<\infty$, \quad then \quad 
	either } 
	\liminf_{t\nearrow\tm} \inf_{x\in\Om} h\big(u(x,t),u_t(x,t)\big)=0
	\quad \mbox{or} \nn\\
	& & \hs{20mm}
	\limsup_{t\nearrow\tm} \Big\{ \|u(\cdot,t)\|_{C^{1+\iota}(\bom)} + \|u_t(\cdot,t)\|_{C^{1+\iota}(\bom)} \Big\}
	= \infty
	\quad \mbox{for all } \iota>0.
  \eea
\end{lem}
\proof
  According to a standard prolongation argument, it is sufficient to make sure that whenever $\iota\in (0,1)$ is such
  that $h\in C^{1+\iota}_{loc}(\R^2)$ and $f\in C^{1+\iota}_{loc}(\R^{2n+2})$, for each $d>0$ and $M>0$ one can find
  $T=T(d,M)\in (0,1)$ with the property that if
  $u_0\in C^2(\bom)$ and $u_{0t}\in C^{1+\iota}(\bom)$ satisfy $u_0=u_{0t}=0$ on $\pO$ as well as
  \be{l01}
	h(u_0,u_{0t})\ge d
	\quad \mbox{in } \bom
	\qquad \mbox{and} \qquad
	\|u_0\|_{C^{1+\iota}(\bom)} + \|u_{0t}\|_{C^{1+\iota}(\bom)} \le M,
  \ee
  a classical solution $u\in C^0([0,T];C^1(\bom)) \cap C^{2,1}(\bom\times (0,T))$
  fulfilling $u_t+au\in C^0(\bom\times [0,T]) \cap C^{2,1}(\bom\times (0,T))$ and 
  $h(u,u_t) > 0$ in $\bom\times (0,T)$ can be found.
%
  To achieve this, given any such $d$ and $M$ we use the continuity of $h$ and the compactness of 
  \be{l02}
	G_0:=\Big\{ (\xi,\sig)\in\R^2 \ \Big| \ |\xi|+|\sig|\le M \mbox{ and } h(\xi,\sig)\ge d\Big\},
  \ee
  as thereby implied, to fix $c_1=c_1(d,M)>0$ and $\eta_1=\eta_1(d,M)\in (0,1]$ such that writing
  \be{l03}
	G:=\Big\{ (\xi,\sig)\in\R^2 \ \Big| \ \inf_{(\xi',\sig')\in G_0} \big\{ |\xi-\xi'|+ |\sig-\sig'|\big\} \le \eta_1\Big\},
  \ee
  we have
  \be{l1}
	\frac{d}{2} \le h(\xi,\sig) \le c_1
	\qquad \mbox{for all } (\xi,\sig)\in G.
  \ee
  As also $h_\xi, h_{\sig}$ and $f$ are continuous, we can therefore find $c_2=c_2(d,M)>0$ and $c_3=c_3(d,M)>0$ such that
  \be{l2}
	\Big| \frac{h_\xi(\xi,\sig)}{h^2(\xi,\sig)}\Big|
	\le c_2
	\quad \mbox{and} \quad
	\Big| \frac{h_{\sig}(\xi,\sig)}{h^2(\xi,\sig)}\Big|
	\le c_2
	\qquad \mbox{for all } (\xi,\sig)\in G,
  \ee
  and that moreover
  \be{l3}
	\Big| \frac{f(\xi,\sig,\Xi,\Sigma)}{h(\xi,\sig)}\Big| \le c_3
	\qquad \mbox{for all $(\xi,\sig)\in G$ and $(\Xi,\Sigma)\in \R^{2n}$ such that } |\Xi|+|\Sigma|\le M_1,
  \ee
  where $M_1:=(1+a)\cdot\big\{(2+a)M+1\big\}+M+1$.
  We next 
  employ a standard result on gradient regularity in scalar parabolic equations (\cite{lieberman}) to pick
  $\theta_0=\theta_0(d,M)\in (0,\iota]$ and $c_4=c_4(d,M)>0$ with the property that whenever $T\in (0,1]$,
  $A_1\in C^1(\Om\times (0,T))$, $A_2\in C^0(\Om\times (0,T);\R^n)$, $A_3\in C^0(\Om\times (0,T))$,
  $A_4\in C^0(\Om\times (0,T))$, $z_0\in C^{1+\iota}(\bom)$ and $z\in C^0(\bom\times [0,T]) \cap L^2((0,T);W_0^{1,2}(\Om))$
  are such that $z_0|_{\pO}=0$ and
  \be{l44}
	\|z_0\|_{C^{1+\iota}(\bom)} \le M
  \ee
  as well as
  \be{l5}
	\frac{2}{d} \le A_1 \le c_1,
	\quad 
	|A_2|\le c_2 M_1,
	\quad
	|A_3| \le a
	\quad \mbox{and} \quad
	|A_4| \le a^2(M+1)+c_3
	\qquad \mbox{in } \Om\times (0,T),
  \ee
  and such that $z$ forms a weak solution, in the standard sense specified in \cite{lieberman}, of the problem
  \be{l6}
	\lball
	z_t = \na \cdot\big(A_1(x,t)\na z\big)
	+ A_2(x,t)\cdot\na z
	+ A_3(x,t) z
	+ A_4(x,t),
	\qquad & x\in\Om, \, t\in (0,T), \\[1mm]
	z(x,t)=0,
	& x\in\pO, \, t\in (0,T), \\[1mm]
	z(x,0)=z_0(x),
	& x\in\Om,
	\ear
  \ee
  it follows that $z\in C^{1+\theta_0,\frac{1+\theta_0}{2}}(\bom\times [0,T])$ with
  \be{l7}
	\|z\|_{C^{1+\theta_0,\frac{1+\theta_0}{2}}(\bom\times [0,T])} \le c_4,
  \ee
  where for definiteness we have set
  \bea{l77}
	\|\vp\|_{C^{1+\theta_0,\frac{1+\theta_0}{2}}(\bom\times [0,T])} 
	&:=&
	\|\vp\|_{L^\infty((0,T);W^{1,\infty}(\Om))}
	+ \sup_{(x,t),(x,s)\in\bom\times [0,T], (x,t)\ne (y,s)} 
		\frac{|\na \vp(x,t)-\na \vp(y,s)|}{|x-y|^{\theta_0} + |t-s|^\frac{\theta_0}{2}} \nn\\
	& & + \sup_{x\in\bom, t,s\in [0,T], t\ne s} \frac{|\vp(x,t)-\vp(x,s)|}{|t-s|^\frac{1+\theta_0}{2}}
  \eea
  for $\vp\in C^{1+\theta_0,\frac{1+\theta_0}{2}}(\bom\times [0,T])$.
  Finally choosing $\eta=\eta(d,M)\in (0,1]$ small enough such that
  \be{l8}
	(2+a)\eta \le\eta_1
  \ee
  and then $T\in (0,1]$ such that
  \be{l81}
	(1-e^{-aT})M + \big\{ 1+(1+a)M\big\}\cdot T \le \eta
  \ee
  and
  \be{l82}
	c_4 T^\frac{\theta_0}{2} \le \eta,
  \ee
  and taking any $\theta=\theta(d,M)\in (0,\theta_0)$, for a fixed pair $(u_0,u_{0t}) \in C^2(\bom)\times C^{1+\iota}(\bom)$
  fulfilling (\ref{l01}) as well as $u_0=u_{0t}=0$ on $\pO$ we introduce the closed subset
  \be{l89}
	S:=\Big\{ \vp\in X \ \Big| \ \|\vp(\cdot,t)-v_0\|_{L^\infty(\Om)} \le \eta \mbox{ and } 
		\|\na\vp(\cdot,t)-\na v_0\|_{L^\infty(\Om)} \le \eta \mbox{ for all } t\in [0,T] \Big\}
  \ee
  of the Banach space $X:=C^{1+\theta,\frac{1+\theta}{2}}(\bom\times [0,T])$, where $v_0:=u_{0t}+au_0$.
  On $S$, we define a mapping $\Phi:S\to X$ by letting $\Phi(\ovv):=v$ for $\ovv\in S$, where for any such $\ovv$ we write
  \be{l9}
	u(x,t):=e^{-at} u_0(x) + \int_0^t e^{-a(t-s)} \ovv(x,s) ds,
	\qquad (x,t)\in\bom\times [0,T],
  \ee
  and let $v$ denote the solution of the linear problem 
  \be{l10}
	\lball
	v_t = \frac{1}{h(u,\ov v-au)} \Del v + av - a^2 u + \frac{f(u,\ovv-au, \na u, \na\ovv-a\na u)}{h(u,\ovv-au)},
	\qquad & x\in\Om, \, t\in (0,T), \\[1mm]
	v(x,t)=0,
	& x\in\pO, \, t\in (0,T), \\[1mm]
	v(x,0)=v_0(x),
	& x\in\Om.
	\ear
  \ee
  In fact, this problem is uniformly parabolic, because for each $x\in\bom$ and $t\in [0,T]$ it follows from the inclusion
  $\ovv\in S$ and (\ref{l9}) that according to the inequality $\eta\le 1$ and our restriction on $T$ in (\ref{l81}) we have
  \bea{l101}
	|u(x,t)-u_0(x)|
	&=& \big| (e^{-at}-1)u_0(x)\big| + \bigg| \int_0^t e^{-a(t-s)} \ovv(x,s) ds \bigg| \nn\\
	&\le& (1-e^{-at}) |u_0(x)|
	+ \int_0^t \big( |\ovv(x,s)-v_0(x)| + |v_0(x)|\big) ds \nn\\
	&\le& (1-e^{-at})M
	+ \big\{ 1+(1+a)M\big\}\cdot T \nn\\
	&\le& \eta.
  \eea
  Therefore, namely, 
  \bas
	& & \hs{-30mm}
	|u(x,t)-u_0(x)|
	+ \big| (\ovv(x,t)-au(x,t)\big) - u_{0t}(x)\big| \\
	&=& |u(x,t)-u_0(x)|
	+ \big| (\ovv(x,t)-v_0(x)\big) - a\big( u(x,t)-u_0(x)\big)\big| \\
	&\le& (1+a) |u(x,t)-u_0(x)| + |\ovv(x,t)-v_0(x)| \\
	&\le& (1+a)\eta +\eta \\
	&\le& \eta_1
	\qquad \mbox{for all $x\in\bom$ and } t\in [0,T]
  \eas
  due to (\ref{l8}), so that since $(u_0(x),u_{0t}(x))\in G_0$ for all $x\in\bom$ by (\ref{l01}), we infer that
  \be{l100}
	\big(u(x,t),\ovv(x,t)-au(x,t)\big) \in G
	\qquad \mbox{for all $x\in\bom$ and } t\in [0,T]
  \ee
  and hence, by (\ref{l1}),
  \be{l11}
	\frac{d}{2} \le h(u,\ovv-au) \le c_1
	\qquad \mbox{in } \bom\times [0,T].
  \ee
  Since from (\ref{l9}) it clearly follows that not only $\ovv$ and $\na\ovv$ but also $u$ and $\na u$ are H\"older continuous
  in $\bom\times [0,T]$, and since thus $h(u,\ovv-au)$ and $f(u,\ovv-au,\na u, \na\ovv-a\na u)$ lie in
  $\bigcup_{\iota'\in (0,1)} C^{\iota',\frac{\iota'}{2}}(\bom\times [0,T])$, in view of the fact that $v_0\in C^{1+\iota}(\bom)$
  with $v_0|_{\pO}=0$ we may hence employ standatd parabolic theory (\cite{LSU}) to confirm that (\ref{l10}) indeed admits a weak
  solution $v\in \bigcup_{\iota'\in (0,1)} C^{\iota',\frac{\iota'}{2}}(\bom\times [0,T]) \cap L^2((0,T);W_0^{1,2}(\Om))$.
  To see that actually $v$ belongs to $S$, we note that since $|\na\ovv| \le \|\na v_0\|_{L^\infty(\Om)}+1$ in $\Om\times (0,T)$
  according to (\ref{l89}) and the inequality $\eta\le 1$, from (\ref{l9}) and the fact that $T\le 1$ it particularly follows that
  \bas
	|\na u(x,t)|
	&=& \bigg| e^{-at} \na u_0(x) + \int_0^t e^{-a(t-s)} \na \ovv(x,s) ds \bigg| \\
	&\le& |\na u_0(x)| + \int_0^t |\na \ovv(x,s)| ds \\
	&\le& \|\na u_{0}\|_{L^\infty(\Om)} + \|\na v_0\|_{L^\infty(\Om)} + 1 \\
	&\le& (2+a)M+1
	\qquad \mbox{for all $x\in\Om$ and } t\in (0,T)
  \eas
  thanks to (\ref{l1}).
  Consequently, for all $x\in\Om$ and $t\in (0,T)$ we have
  \bea{l12}
	|\na u(x,t)| + |\na\ovv(x,t)-a\na u(x,t)|
	&\le& (1+a) |\na u(x,t)| + |\na\ovv(x,t)| \nn\\
	&\le& (1+a)\cdot\big\{ (2+a) M +1\big\} + M+1
	= M_1,
  \eea
  so that in line with (\ref{l100}) we may draw on (\ref{l3}) to see that for
  \bas
	A_4(x,t):=-a^2 u + \frac{f(u,\ovv-au,\na u,\na \ovv-a\na u)}{h(u,\ovv-au)},
	\qquad (x,t)\in\Om\times (0,T),
  \eas
  we have
  $|A_4(x,t)| \le a^2(M+1)+c_3$ for all $x\in\Om$ and $t\in (0,T)$, because
  $|u(x,t)| \le |u_0(x)| + \eta \le M+1$ for all $x\in\Om$ and $t\in (0,T)$ by (\ref{l101}), (\ref{l01}) and again the restriction
  $\eta\le 1$.
  Since from (\ref{l11}), (\ref{l100}), (\ref{l2}) and (\ref{l12}) we furthermore know that
  \bas
	A_1(x,t):=\frac{1}{h(u,\ovv-au)}
	\quad \mbox{and} \quad
	A_3(x,t):=a,
	\qquad (x,t)\in\Om\times (0,T),
  \eas
  as well as
  \bas
	A_2(x,t):=\na \frac{1}{h(u,\ovv-au)}
	\equiv - \frac{h_u(u,\ovv-au)}{h^2(u,\ovv-au)} \na u 
	+ \frac{h_{u_t}(u,\ovv-au)}{h^2(u,\ovv-au)} \na (\ovv-au),
	\qquad (x,t)\in\Om\times (0,T),
  \eas
  satisfy
  \bas
	\frac{2}{d} \le A_1 \le c_1
	\quad \mbox{and} \quad
	|A_3|\le a
	\qquad \mbox{in } \Om\times (0,T)
  \eas
  as well as
  \bas
	|A_2|
	\le c_2 |\na u| + c_2 |\na\ovv-a\na u|
	\le c_2 M_1
	\qquad \mbox{in } \Om\times (0,T).
  \eas
  As moreover
  \bas
	\|v_0\|_{C^{1+\iota}(\bom)} = \|u_{0t}+au_0\|_{C^{1+\iota}(\bom)}
	\le (1+a)M
  \eas
  due to (\ref{l01}), we thus see that all the conditions in (\ref{l44}) and (\ref{l5}) are satisfied with $z_0:=v_0$, whence 
  we may draw on (\ref{l7}) to conclude that $v\in C^{1+\theta_0,\frac{1+\theta_0}{2}}(\bom\times [0,T])$, with
  $\|v\|_{C^{1+\theta_0,\frac{1+\theta_0}{2}}(\bom\times [0,T])} \le c_4$.
  As $\theta\le\theta_0$, this especially means that $v\in X$, and using the information on continuity in time
  correspondingly contained in (\ref{l7}) we particularly see that
  \bas
	|v(x,t)-v_0(x)|
	\le c_4 t^\frac{\theta_0}{2}
	\le \eta
	\quad \mbox{and} \quad
	|\na v(x,t)-\na v_0(x)| \le c_4 t^\frac{\theta_0}{2}
	\le \eta
	\qquad \mbox{for all $x\in\Om$ and } t\in (0,T)
  \eas
  according to (\ref{l77}) and (\ref{l82}).
  In consequence, $\Phi$ maps $S$ into itself, and making use of the strict inequality $\theta<\theta_0$ now, we may rely on
  the Arzel\`a-Ascoli theorem to see that (\ref{l7}) actually ensures that $\ov{\Phi(S)}$ is a compact subset of $X$.\abs
  To verify that $\Phi$ is continuous, we fix $(\ovv_i)_{j\in\N} \subset S$ and $\ovv\in S$ such that
  $\ovv_j \to \ovv$ in $X$ as $j\to\infty$, and then firstly obtain that the functions $u_j$ and $u$ accordingly defined through
  (\ref{l9}) also satisfy $u_j\to u$ in $X$ as $j\to\infty$.
  In particular, this implies that $\frac{1}{h(u_j,\ovv_j-au_j)} \to \frac{1}{h(u,\ovv-au)}$, 
  $\na \frac{1}{h(u_j,\ovv_j-au_j)} \to \na \frac{1}{h(u,\ovv-au)}$ and
  $f(u_j,\ovv_j-au_j,\na u_j,\na\ovv_j-a\na u_j) \to f(u,\ovv-au,\na u,\na\ovv-a\na u)$ in $L^\infty(\Om\times (0,T))$
  as $j\to\infty$, whence according to a standard uniqueness property of the problem class (\ref{l6}) 
  (\cite{LSU}), each accumulation point of $(\Phi(\ovv_j))_{j\in\N}$ must coincide with the solution of (\ref{l10}).
  Since $(\Phi(\ovv_j))_{j\in\N}$ is relatively compact in $X$ due to an estimate of the form in (\ref{l7}) and the inequality
  $\theta<\theta_0$, this means that, indeed, $\Phi(\ovv_j)\to\Phi(\ovv)$ in $X$ as $j\to\infty$.\abs
  Noting that $S$ is convex, we may therefore employ the Schauder fixed point theorem to infer the existence of $v\in S$
  such that $\Phi(v)=v$, and straightforward bootstrap-type arguments involving parabolic regularity theory (\cite{LSU})
  finally reveal that $v$, along with the quantity $u$ accordingly defined by (\ref{l9}), in fact forms a pair $(u,v)$
  of functions which both belong to $C^0([0,T];C^1(\bom)) \cap C^{1,2}(\bom\times (0,T))$, which satisfy
  $h(u,v-au)>0$ in $\bom\times [0,T]$, and which solve (\ref{0v}) in the classical sense in $\Om\times (0,T)$.
  The proof thereby becomes complete.
\qed
By means of parabolic regularity theory, the extensibility criterion in (\ref{ext}) can considerably be sharpened:
\begin{lem}\label{lem17}
  Let $h\in \bigcup_{\iota>0} C^{1+\iota}_{loc} (\R^2)$ and
  $f\in \bigcup_{\iota>0} C^{1+\iota}_{loc} (\R^{2n+2})$ be such that there exists $\lam>1$ with the property that for each $M>0$
  one can find $K_f(M)>0$ such that (\ref{f1}) holds.
  Then whenever $u_0\in C^2(\bom)$ and $u_{0t} \in \bigcup_{\iota>0} C^{1+\iota}(\bom)$ satisfy $u_0=u_{0t}=0$ on $\pO$,
  for $\tm$ and $u$ as correspondingly obtained in Lemma \ref{lem_loc} it follows that
  \bea{ext1}
	\hs{-2mm}
	\mbox{if $\tm<\infty$, \ then \  
	either \ } 
	\liminf_{t\nearrow\tm} \inf_{x\in\Om} h\big(u(x,t),u_t(x,t)\big)=0
	\ \mbox{or} \
	\limsup_{t\nearrow\tm} \|u_t(\cdot,t)\|_{W^{1,p}(\Om)} 
	= \infty
  \eea
  for all $p>n\cdot\max\{\lam,2\}$ fulfilling $p\ge 2\cdot\max\{\lam,2\}$.
\end{lem}
\proof
  Supposing on the contrary that $\tm$ be finite, but that $h(u,u_t)\ge c_1$ and 
  \be{17.01}
	\|u_t\|_{W^{1,p}(\Om)} \le c_2
	\qquad \mbox{for all } t\in (0,\tm)
  \ee
  with some $c_1>0$ and $c_2>0$ and some $p>n\cdot\max\{\lam,2\}$ such that $p\ge 2\cdot\max\{\lam,2\}$
  we could firstly conclude that for all $t\in (0,\tm)$,
  \be{17.02}
	\|u(\cdot,t)\|_{W^{1,p}(\Om)}
	= \bigg\| u_0 + \int_0^t u_t(\cdot,s) ds \bigg\|_{W^{1,p}(\Om)}
	\le c_3:=\|u_0\|_{W^{1,p}(\Om)} + c_2 \tm,
  \ee
  and that thus, by continuity of the embedding $W^{1,p}(\Om) \hra L^\infty(\Om)$ asserted by the inequality $p>n$, 
  \be{17.03}
	\|u\|_{L^\infty(\Om)}
	+ \|u_t\|_{L^\infty(\Om)}
	\le c_4
	\qquad \mbox{for all } t\in (0,\tm)
  \ee
  with some $c_4>0$.
  By continuity of $h$, $h_u$ and $h_{u_t}$, this would entail the existence of positive constants $c_5, c_6$ and $c_7$ such that
  \bas
	c_1 \le h(u,u_t) \le c_5,
	\quad 
	\Big| \frac{h_u(u,u_t)}{h^2(u,u_t)} \Big| \le c_6
	\quad \mbox{and} \quad
	\Big| \frac{h_{u_t}(u,u_t)}{h^2(u,u_t)} \Big| \le c_7
	\qquad \mbox{in } \Om\times (0,\tm),
  \eas
  while on the basis of (\ref{17.03}), an application of (\ref{f1}) would provide $c_8>0$ satisfying
  \bas
	\Big| \frac{f(u,u_t,\na u,\na u_t)}{h(u,u_t)} \Big|
	\le c_8\cdot \big(1+|\na u|+ |\na u_t|\big)^\lam
	\qquad \mbox{in } \Om\times (0,\tm).
  \eas
  In the identity
  \be{17.1}
	v_t = \na\cdot\big(A_1(x,t)\na v\big) + \wh{A}_1(x,t),
	\qquad x\in\Om, \ t\in (0,\tm),
  \ee
  according to (\ref{0}) valid for $v:=u_t+au$ with
  \be{17.2}
	A_1(x,t):=\frac{1}{h(u,u_t)}
	\quad \mbox{and} \quad
	\wh{A}_1(x,t):=\na A_1 \cdot \na (u_t+au) + au_t + \frac{f(u,u_t,\na u,\na u_t)}{h(u,u_t)},
	\ (x,t)\in \Om\times (0,\tm),
  \ee
  we thus could estimate
  \be{17.3}
	\frac{1}{c_5} \le A_1 \le \frac{1}{c_1}
	\qquad \mbox{in } \Om\times (0,\tm)
  \ee
  and 
  \be{17.33}
	|\na A_1|
	= \Big| - \frac{h_u(u,u_t)}{h^2(u,u_t)} \na u
	- \frac{h_{u_t}(u,u_t)}{h^2(u,u_t)} \na u_t \Big|
	\le c_6 |\na u| + c_7 |\na u_t|
	\qquad \mbox{in } \Om\times (0,\tm)
  \ee
  as well as
  \be{17.4}
	|\wh{A}_1|
	\le c_9\cdot \big( 1+ |\na u|^\kappa +  |\na u_t|^\kappa\big)
	\qquad \mbox{in } \Om\times (0,\tm)
  \ee
  with $\kappa:=\max\{\lam,2\}$ and some $c_9>0$.
  Therefore, writing $q:=\frac{p}{\kappa}$ and recalling (\ref{17.01}) and (\ref{17.02}), we would obtain $c_{10}>0$ and
  $c_{11}>0$ such that
  \be{17.5}
	\io |\wh{A}_1|^q
	\le c_{10} + c_{10} \io |\na u|^{q\kappa}
	+ c_{10} \io |\na u_t|^{q\kappa} 
	\le c_{11}
	\qquad \mbox{for all } t\in (0,\tm),
  \ee
  and, moreover, from (\ref{17.33}), (\ref{17.01}) and (\ref{17.02}) and the continuity of the embedding
  $W^{1,p}(\Om) \hra C^{\theta_1}(\bom)$, with $\theta_1:=1-\frac{n}{p}$ being positive since $p>n$, we could confirm that
  \be{17.6}
	\|A_1\|_{C^{\theta_1}(\bom)} \le c_{12}
	\qquad \mbox{for all } t\in (0,\tm)
  \ee
  with some $c_{12}>0$. 
  Based on (\ref{17.3}), (\ref{17.5}) and (\ref{17.6}), we could thus rely on the fact that $q>n$ and $q\ge 2$,
  as guaranteed by our hypotheses $p>n\cdot\max\{\lam,2\}$ and $p\ge 2\cdot\max\{\lam,2\}$,
  in applying a classical result on gradient regularity in parabolic equations (\cite{lieberman}) to find
  $\theta_2\in (0,1)$ such that $v\in C^{1+\theta_2,\frac{1+\theta_2}{2}}(\bom\times [0,\tm])$.
  In particular, this would provide $c_{13}>0$ such that $\|v\|_{C^{1+\theta_2}(\bom)} \le c_{13}$ for all $t\in (0,\tm)$,
  whence again using the identity $u_t=v-au$ we could infer the existence of $c_{14}>0$ such that
  $\|u\|_{C^{1+\theta_2}(\bom)} + \|u_t\|_{C^{1+\theta_2}(\bom)} \le c_{14}$ for all $t\in (0,\tm)$.
  In view of the contradiction to (\ref{ext}) thereby achieved, we may thus conclude as intended.
\qed
Throughout the sequel, we let $\tm$ and $u$ be as provided by Lemma \ref{lem_loc}, and
in order to suitably exploit the fundamental structure of (\ref{0}) with regard to the action of second-order spatial
operators appearing therein, in what follows we let
\be{v}
	v:=u_t+au,
\ee
noting that then, by (\ref{0}),
\be{0v}
	h(u,u_t)v_t = \Del v + ah(u,u_t)u_t + f(u,u_t,\na u, \na u_t)
	\qquad \mbox{in } \Om\times (0,\tm)
\ee
with $v=0$ on $\pO\times (0,\tm)$.
\mysection{A low-regularity enegy functional featuring linear dissipation}\label{sect4}
%
%
%
%
%
%
%
%
In this section aiming at the derivation of a linearly damped ODI of the form (\ref{08}), let us first record
some elementary properties of the ingredients thereof, as appearing in (\ref{06}) and (\ref{07}).
\begin{lem}\label{lem11}
  Assume (\ref{h}).
  Then there exist $\del_1=\del_1(h)\in (0,1]$ and $\Gamma_1=\Gamma_1(h)>0$ such that writing
  \be{H1}
	H^{(1)}(\xi,\sig) := \int_0^\sig \xi h(\xi,\wh{\sig}) d\wh{\sig},
	\qquad (\xi,\sig)\in\R^2,
  \ee
  and 
  \be{H2}
	H^{(2)}(\xi,\sig):=\int_0^\sig \wh{\sig} h(\xi,\wh{\sig})d\wh{\sig},
	\qquad (\xi,\sig)\in\R^2,
  \ee
  we have
  \be{11.1}
	\frac{h_0}{2} \le h(\xi,\sig) \le 2h_0
	\qquad \mbox{for all } (\xi,\sig)\in [-\del_1,\del_1]^2
  \ee
  and
  \be{11.2}
	\big| H^{(1)}(\xi,\sig)\big| 
	\le 2h_0 |\xi|\cdot |\sig|
	\qquad \mbox{for all } (\xi,\sig)\in [-\del_1,\del_1]^2
  \ee
  and
  \be{11.3}
	\frac{h_0}{4} \sig^2 \le H^{(2)}(\xi,\sig) \le h_0 \sig^2
	\qquad \mbox{for all } (\xi,\sig)\in [-\del_1,\del_1]^2
  \ee
  as well as
  \be{11.4}
	\big| H_{\xi}^{(1)}(\xi,\sig)\big| \le \Gamma_1 |\sig|
	\qquad \mbox{for all } (\xi,\sig)\in [-\del_1,\del_1]^2
  \ee
  and
  \be{11.5}
	\big|H_{\xi}^{(2)}(\xi,\sig)\big| \le \Gamma_1 \sig^2
	\qquad \mbox{for all } (\xi,\sig)\in [-\del_1,\del_1]^2,
  \ee
  where $h_0:=h(0,0)$.
\end{lem}
\proof
  By continuity of $h$ and the positivity of $h_0$, the existence of $\del_1=\del_1(h)\in (0,1]$ fulfilling (\ref{11.1})
  is evident. From (\ref{H1}) and (\ref{H2}) we thereupon readily obtain both (\ref{11.2}) and (\ref{11.3}), while using
  (\ref{11.1}) to estimate
  \bas
	\big| H^{(1)}_\xi(\xi,\sig)\big|
	&=& \bigg| \int_0^\sig h(\xi,\wh{\sig}) d\wh{\sig} + \int_0^\sig \xi h_\xi(\xi,\wh{\sig}) d\wh{\sig} \bigg| \\
	&\le& 2h_0 |\sig|
	+ \|h_\xi\|_{L^\infty((-\del_1,\del_1))} \cdot |\xi|\cdot |\sig|
  \eas
  and
  \bas
	\big| H^{(2)}_\xi(\xi,\sig)\big|
	= \bigg| \int_0^\sig \wh{\sig} h_\xi(\xi,\wh{\sig})d\wh{\sig} \bigg|
	\le \|h_\xi\|_{L^\infty((-\del_1,\del_1))} \cdot \frac{\sig^2}{2}
  \eas
  for $(\xi,\sig) \in [-\del_1,\del_1]^2$, we conclude that also (\ref{11.4}) and (\ref{11.5}) hold if we let
  $\Gamma_1\equiv \Gamma_1(h):=
  \max \Big\{ 2h_0+\|h_\xi\|_{L^\infty((-\del_1,\del_1))} \cdot \del_1 , \frac{1}{2} \|h_\xi\|_{L^\infty((-\del_1,\del_1))} \Big\}$.
\qed
In view of the mere definitions of $H^{(1)}$ and $H^{(2)}$, two standard evolution features of the viscous wave equation (\ref{0})
immediately follow from basic testing procedures:
\begin{lem}\label{lem12}
  Let $H^{(1)}$ and $H^{(2)}$ be as in Lemma \ref{lem11}. Then
  \bea{12.1}
	& & \hs{-20mm}
	\frac{d}{dt} \bigg\{ \frac{1}{2} \io |\na u|^2
	+ \io H^{(1)}(u,u_t) \bigg\} + a \io |\na u|^2 \nn\\
	&=& \io f(u,u_t,\na u,\na u_t) u
	+ \io H^{(1)}_\xi(u,u_t) u_t
  \eea
  and
  \bea{12.2}
	& & \hs{-20mm}
	\frac{d}{dt} \bigg\{ \io H^{(2)}(u,u_t) + \frac{a}{2} \io |\na u|^2 \bigg\}
	+ \io |\na u_t|^2 \nn\\
	&=& \io H^{(2)}_\xi(u,u_t) u_t
	+ \io f(u,u_t,\na u,\na u_t) u_t
  \eea
  for all $t\in (0,\tm)$.
\end{lem}
\proof
  Since $H^{(1)}_\sig(\xi,\sig)=\xi h(\xi,\sig)$ for $(\xi,\sig)\in\R^2$ by (\ref{H1}), using (\ref{0}) {\cb and the chain rule} we find that
  for all $t\in (0,\tm)$,
  \bas
	\frac{d}{dt} \io H^{(1)}(u,u_t)
	&=& \io H^{(1)}_\sig(u,u_t) u_{tt}
	+ \io H^{(1)}_\xi(u,u_t) u_t \\
	&=& \io uh(u,u_t) u_{tt} 
	+ \io H^{(1)}_\xi(u,u_t) u_t \\
	&=& \io u\cdot \Big\{ \Del u_t + a\Del u + f(u,u_t,\na u,\na u_t)\Big\}
	+ \io H^{(1)}_\xi(u,u_t) u_t,
  \eas
  from which (\ref{12.1}) results in a straightforward manner upo ingetrating by parts.\abs
  Likewise, combining (\ref{0}) with the identity $H^{(2)}_\sig(\xi,\sig)=\sig h(\xi,\sig)$, as asserted by (\ref{H2}) 
  for $(\xi,\sig)\in\R^2$, shows that
  \bas
	\frac{d}{dt} \io H^{(2)}(u,u_t)
	&=& \io H^{(2)}_\sig(u,u_t) u_{tt}
	+ \io H^{(2)}_\xi(u,u_t) u_t \\
	&=& \io u_t h(u,u_t) u_{tt}
	+ \io H^{(2)}_\xi(u,u_t) u_t \\
	&=& \io u_t \cdot \Big\{ \Del u_t + a\Del u + f(u,u_t,\na u,\na u_t)\Big\}
	+ \io H^{(2)}_\xi(u,u_t) u_t
  \eas
  for all $t\in (0,\tm)$, whence also (\ref{12.2}) holds.
\qed
Now under an {\em a priori} assumption of the form in (\ref{05}), we may rely on Lemma \ref{lem11}
to obtain the following by suitable linear combination of the identities in (\ref{12.1}) and (\ref{12.2}).
\begin{lem}\label{lem13}
  Let $\del_1$ be as in Lemma \ref{lem11}. Then there exist $B=B(a,h)>0$, $\Gamma_2=\Gamma_2(a,h)>0$ and $\Gamma_3=\Gamma_3(a,h)>0$
  such that if (\ref{init}) holds, and if $T\in (0,\tm]$ is such that
  \be{uut}
	|u| + |u_t| \le \del_1
	\qquad \mbox{in } \Om\times (0,T),
  \ee
  then for
  \be{F1}
	\F_1(t):=\frac{1+Ba}{2} \io |\na u|^2
	+ \io H^{(1)}(u,u_t) + B \io H^{(2)}(u,u_t),
	\qquad t\in (0,\tm),
  \ee
  we have
  \be{13.1}
	\Gamma_2 \io v^2 \le \F_1(t) \le \Gamma_3 \io |\na u|^2 + \Gamma_3 \io |\na u_t|^2
	\qquad \mbox{for all } t\in [0,T)
  \ee
  and 
  \bea{13.2}
	\F_1'(t) + \Gamma_2 \F_1(t)
	\le \Gamma_3 \io |u|\cdot |f(u,u_t,\na u,\na u_t)|
	+ \Gamma_3 \io |u_t|\cdot |f(u,u_t,\na u,\na u_t)|
	+ \Gamma_3 \io |u_t|^3
  \eea
  for all $t\in (0,T)$.
\end{lem}
\proof
  According to a Poincar\'e inequality, we take $c_1>0$ such that
  \be{13.3}
	c_1 \io \vp^2 \le \io |\na\vp|^2
	\qquad \mbox{for all } \vp\in W_0^{1,2}(\Om),
  \ee
  and with $h_0$ and $\Gamma_1=\Gamma_1(h)$ as in Lemma \ref{lem11}, we fix $B=B(a,h)>0$ large enough fulfilling
  \be{13.4}
  	B\ge \max\bigg\{\frac{17h_0}{c_1},  \frac{ 2\Gamma_1}{c_1}\bigg\}
  \ee
  Now assuming (\ref{init}) and (\ref{uut}) to hold with some $T\in (0,\tm]$, we may draw on (\ref{11.2}) and (\ref{11.3}) to see that
  thanks to Young's inequqlity and (\ref{13.3}),
  \bea{13.5}
	\F_1(t)
	&\ge& \frac{1+Ba}{2} \io |\na u|^2
	- 2h_0 \io |u|\cdot |u_t| 
	+ \frac{B h_0}{4} \io u_t^2 \nn\\
	&\ge& \frac{1+Ba}{2} \io |\na u|^2
	- \frac{c_1}{4} \io u^2
	+ \Big\{ \frac{B h_0}{4} - \frac{4h_0^2}{c_1}\Big\} \cdot \io u_t^2 \nn\\
	&\ge& \Big\{ \frac{(1+Ba)c_1}{2} - \frac{c_1}{4}\Big\} \cdot \io u^2
	+ \Big\{ \frac{B h_0}{4} - \frac{4h_0^2}{c_1}\Big\} \cdot \io u_t^2 \nn\\
	&\ge& c_2 \io u^2 + c_3 \io u_t^2
	\qquad \mbox{for all } t\in (0,T),
  \eea
  where the first inequality in (\ref{13.4}) ensures that not only $c_2\equiv c_2(a,h):=\frac{(1+Ba)c_1}{2} - \frac{c_1}{4}$
  but also $c_3\equiv c_3(a,h):=\frac{B h_0}{4} - \frac{4h_0^2}{c_1}$ is positive.
  Since, on the other hand, in view of (\ref{v}) and Young's inequality we have 
  \bas
	\io v^2 = \io (u_t+au)^2 \le 2\io u_t^2 + 2a^2 \io u^2
	\qquad \mbox{for all } t\in (0,T),
  \eas
  from (\ref{13.5}) we particularly obtain that
  \be{13.6}
	\F_1(t) \ge \min \Big\{ \frac{c_2}{2a^2} \, , \, \frac{c_3}{2}\Big\} \cdot \io v^2
	\qquad \mbox{for all } t\in (0,T).
  \ee
  We next combine (\ref{12.1}) with (\ref{12.2}) to compute
  \bea{13.7}
	& & \hs{-20mm}
	\F_1'(t)
	+ a \io |\na u|^2
	+ B \io |\na u_t|^2 \nn\\
	&=& \io f(u,u_t,\na u,\na u_t) u + \io H^{(1)}_\xi(u,u_t) u_t \nn\\
	& & + B \io H^{(2)}_\xi(u,u_t) u_t
	+ B \io f(u,u_t,\na u,\na u_t) u_t
	\qquad \mbox{for all } t\in (0,T),
  \eea
  and here we combine (\ref{11.1}) and (\ref{11.4}) with (\ref{13.3}) and the second inequality in (\ref{13.4}) to estimate
  \bea{13.8}
	 \io H^{(1)}_\xi(u,u_t) u_t
	&\le& \Gamma_1 \io u_t^2 \nn\\
	&\le& \frac{\Gamma_1}{c_1} \io |\na u_t|^2 \nn\\
	&\le& \frac{B}{2} \io |\na u_t|^2
	\qquad \mbox{for all } t\in (0,T).
  \eea
  Apart from that, from (\ref{11.5}) we know that
  \be{13.9}
	B \io H^{(2)}_\xi(u,u_t)u_{t} \le B\Gamma_1 \io |u_t|^3
	\qquad \mbox{for all } t\in (0,T),
  \ee
  while again due to (\ref{11.2}), (\ref{11.3}), Young's inequality and (\ref{13.3}),
  \bea{13.10}
	\F_1(t)
	&\le& \frac{1+Ba}{2} \io |\na u|^2
	+ 2h_0 \io |u|\cdot |u_t|
	+ Bh_0 \io u_t^2 \nn\\
	&\le& \frac{1+Ba}{2} \io |\na u|^2
	+ h_0 \io u^2
	+ (1+B)h_0 \io u_t^2 \nn\\
	&\le& \Big\{ 1 + \frac{Ba}{2} + \frac{h_0}{c_1}\Big\} \cdot \io |\na u|^2
	+ \frac{(1+B)h_0}{c_1} \io |\na u_t|^2
	\qquad \mbox{for all } t\in (0,T).
  \eea
  If we let $c_4\equiv c_4(a,h):=
  \min\Big\{ a\cdot \big\{1+\frac{Ba}{2}+\frac{h_0}{c_1}\big\}^{-1} \, , \, \frac{B}{2}\cdot \big\{\frac{(1+B)h_0}{c_1}\big\}^{-1}
  \Big\}$,
  then from (\ref{13.7})-(\ref{13.10}) we therefore obtain that
  \bas
	\F_1'(t) + c_4 \F_1(t) \le \io f(u,u_t,\na u,\na u_t) u
	+ B \io f(u,u_t,\na u, \na u_t) u_t
	+ B\Gamma_1 \io |u_t|^3
	\qquad \mbox{for all } t\in (0,T),
  \eas
  which together with (\ref{13.6}) and (\ref{13.10}) shows that both (\ref{13.1}) and (\ref{13.2}) hold if
  $\Gamma_2=\Gamma_2(a,h)$ is sufficiently small and $\Gamma_3=\Gamma_3(a,h)$ is suitably large.
\qed
The basic energy inequality in (\ref{13.2}) will subsequently be used through the following slightly modified variant thereof.
\begin{cor}\label{cor14}
  Let $p\ge 2$. Then there exist $\Gamma_4=\Gamma_4(p,a,h)>0$ and $\Gamma_5=\Gamma_5(p,a,h)>0$ such that if (\ref{init}) 
  and (\ref{uut}) hold with $\del_1=\del_1(h)$ as in Lemma \ref{lem11} and with some $T\in (0,\tm]$, then the function $\F_1$
  from (\ref{F1}) satisfies
  \bea{14.1}
	& & \hs{-20mm}
	\frac{d}{dt} \F_1^\frac{p}{2}(t)
	+ \Gamma_4 \cdot \F_1^\frac{p}{2}(t)
	\nn\\
	&\le& \Gamma_5 \cdot \bigg\{ \io |u| \cdot |f(u,u_t,\na u, \na u_t)| \bigg\}^\frac{p}{2}
	+ \Gamma_5 \cdot \bigg\{ \io |v|\cdot |f(u,u_t,\na u,\na u_t)|\bigg\}^\frac{p}{2} \nn\\
	& & + \Gamma_5 \cdot \bigg\{ \io |u|^3 \bigg\}^\frac{p}{2}
	+ \Gamma_5 \cdot \bigg\{ \io |v|^3 \bigg\}^\frac{p}{2}
	\qquad \mbox{for all } t\in (0,T).
  \eea
\end{cor}
\proof
  Suppressing the argument $(u,u_t,\na u,\na u_t)$ of $f$, in (\ref{13.2}) we use (\ref{v}) and Young's inequality to estimate
  \bea{14.2}
	& & \hs{-20mm}
	\Gamma_3 \io |u|\cdot |f|
	+ \Gamma_3 \io |u_t|\cdot |f|
	+ \Gamma_3 \io |u_t|^3 \nn\\
	&\le& I(t)
	:= (1+a) \Gamma_3 \io |u| \cdot |f|
	+ \Gamma_3 \io |v| \cdot |f|
	+ 4a^3 \Gamma_3 \io |u|^3
	+ 4\Gamma_3 \io |v|^3
  \eea
  for all $t\in (0,T)$, so that since $p\ge 2$, again by means of Young's inequality we infer from Lemma \ref{lem13} that
  \bas
	\frac{d}{dt} \F_1^\frac{p}{2}(t)
	&\le& \frac{p}{2} \F_1^\frac{p-2}{2}(t) \cdot \big\{ - \Gamma_2 \F_1(t) + I(t)\big\} \\
	&=& - \frac{p\Gamma_2}{2} \F_1^\frac{p}{2}(t)
	+ \Big\{ \frac{p\Gamma_2}{4} \F_1^\frac{p}{2}(t)\Big\}^\frac{p-2}{p} \cdot
		\big(\frac{4}{p\Gamma_2}\Big)^\frac{p-2}{p} \cdot \frac{p}{2} I(t) \\
	&\le& - \frac{p\Gamma_2}{4} \F_1^\frac{p}{2}(t)
	+ \Big(\frac{4}{p\Gamma_2}\Big)^\frac{p-2}{p} \cdot \Big(\frac{p}{2}\Big)^\frac{p}{2} I^\frac{p}{2}(t)
	\qquad \mbox{for all } t\in (0,T).
  \eas
  In view of (\ref{14.2}), after another application of Young's inequality this yields (\ref{14.1}) upon evident choices
  of $\Gamma_i=\Gamma_i(p,a,h)$ for $i\in\{4,5\}$.
\qed
\mysection{A higher-order energy functional}\label{sect5}
This section will be devoted to the core of our analysis, consisting in the inequality announced in (\ref{09})-(\ref{010}).
As a preliminary for an associated testing procedure, let is import from \cite[Lemma 3.4]{black_win_M3AS} 
a result from elementary differential calculus which provides a one-sided pointwise estimate for the boundary derivatives of gradients of functions satisfying homogeneous Dirichlet conditions,
and which will be used to control the boundary integrals in Lemma \ref{lem1}. 
The following has been imported from \cite[Lemma 3.4]{black_win_M3AS}.
\begin{lem}\label{lem_black}
  Let $\kappa\in\R$ denote the maximum of the curvatures on $\pO$. 
  Then for each $\vp\in C^2(\bom)$ fulfilling $\vp=0$ on $\pO$,
  \bas
	\frac{\partial |\na\vp|^2}{\pa \nu}
	\le 2\frac{\pa \vp}{\pa\nu} \Del \vp
	+ 2\kappa \Big| \frac{\pa\vp}{\pa\nu}\Big|^2
	\qquad \mbox{on } \pO.
  \eas
\end{lem}
We can thereby describe a basic evolution property of the spatial $W^{1,p}$ norm of the function introduced in (\ref{v}):
\begin{lem}\label{lem1}
  Assume (\ref{init}), and suppose that $T\in (0,\tm]$ is such that
  \be{1.1}
	h(u,u_t)>0
	\qquad \mbox{in } \bom\times [0,T).
  \ee
  Then whenever $p\ge 2$,
  \bea{1.2}
	& & \hs{-12mm}
	\frac{1}{p} \frac{d}{dt} \io |\na v|^p
	+ \io \frac{1}{h(u,u_t)} |\na v|^{p-2} |D^2 v|^2 \nn\\
	&\le& \io \frac{h_u(u,u_t)}{h^2(u,u_t)} |\na v|^{p-2} \cdot \Big\{ \na u \cdot (D^2 v\cdot\na v) - (\na u\cdot\na v)\Del v
		\Big\} \nn\\
	& & + \io \frac{h_{u_t}(u,u_t)}{h^2(u,u_t)} |\na v|^{p-2} \cdot 
		\Big\{ \na v\cdot (D^2 v\cdot\na v) - a\na u\cdot (D^2 v\cdot\na v) 
		- |\na v|^2 \Del v + a(\na u\cdot\na v)\Del v \Big\} \nn\\
	& & + a \io |\na v|^p
	- a^2 \io |\na v|^{p-2} \na u\cdot\na v \nn\\
	& & - \io \frac{f(u,u_t,\na u,\na u_t)}{h(u,u_t)} |\na v|^{p-4} \cdot 
		\Big\{ (p-2)\na v\cdot (D^2 v\cdot\na v) + |\na v|^2 \Del v\Big\} \nn\\
	& & + \kappa \int_{\pO} \frac{1}{h(u,u_t)} |\na v|^p
	\qquad \mbox{for all } t\in (0,T),
  \eea
  where $\kappa$ is as in Lemma \ref{lem_black}.
\end{lem}
\proof
  As a consequence of (\ref{0v}) and two integrations by parts, we obtain that
  \bea{1.3}
	\frac{1}{p} \frac{d}{dt} \io |\na v|^p
	&=& \io |\na v|^{p-2} \na v\cdot 
		\na \Big\{ \frac{1}{h(u,u_t)} \Del v + av - a^2 u + \frac{f(u,u_t,\na u,\na u_t)}{h(u,u_t)}\Big\} \nn\\
	&=& \io \frac{1}{h(u,u_t)} |\na v|^{p-2} \na v\cdot\na\Del v
	+ \io |\na v|^{p-2} \Big\{ \na v\cdot\na \frac{1}{h(u,u_t)} \Big\} \Del v \nn\\
	& & + a \io |\na v|^p
	- a^2 \io |\na v|^{p-2} \na u\cdot\na v \nn\\
	& & - \io \frac{f(u,u_t,\na u,\na u_t)}{h(u,u_t)} \na \cdot \big\{ |\na v|^{p-2} \na v\big\} \nn\\
	& & + \int_{\pO} \frac{f(u,u_t,\na u,\na u_t)}{h(u,u_t)} |\na v|^{p-2} \frac{\pa v}{\pa\nu}
	\qquad \mbox{for all } t\in (0,T).
  \eea
  Here, since $\na v\cdot\na\Del v = \frac{1}{2} \Del |\na v|^2 - |D^2 v|^2$ and
  \be{1.4}
	\na |\na v|^{p-2} = \frac{p-2}{2} |\na v|^{p-4} \na |\na v|^2,
  \ee
  another integration by parts shows that
  \bea{1.5}
	& & \hs{-20mm}
	\io \frac{1}{h(u,u_t)} |\na v|^{p-2} \na v\cdot \na\Del v \nn\\
	&=& \frac{1}{2} \io \frac{1}{h(u,u_t)} |\na v|^{p-2} \Del |\na v|^2
	- \io \frac{1}{h(u,u_t)} |\na v|^{p-2} |D^2 v|^2 \nn\\
	&=& - \frac{p-2}{4} \io \frac{1}{h(u,u_t)} |\na v|^{p-4} \Big|\na |\na v|^2\Big|^2 
	- \frac{1}{2} \io |\na v|^{p-2} \na |\na v|^2 \cdot \na \frac{1}{h(u,u_t)} \nn\\
	& & + \frac{1}{2} \int_{\pO} \frac{1}{h(u,u_t)} |\na v|^{p-2} \frac{\pa |\na v|^2}{\pa\nu}
	- \io \frac{1}{h(u,u_t)} |\na v|^{p-2} |D^2 v|^2
	\qquad \mbox{for all } t\in (0,T).
  \eea
  Since in view of the identity $u_t=v-au$ we have
  \bas
	\na \frac{1}{h(u,u_t)}
	= - \frac{h_u(u,u_t)}{h^2(u,u_t)} \na u
	- \frac{h_{u_t}(u,u_t)}{h^2(u,u_t)} \na v
	+ a \frac{h_{u_t}(u,u_t)}{h^2(u,u_t)} \na u
	\qquad \mbox{in } \Om\times (0,T),
  \eas
  since $\na |\na v|^2 = 2D^2 v \cdot\na v$, and since thus
  \bas
	\na\cdot \big\{ |\na v|^{p-2} \na v\big\}
	= (p-2) |\na v|^{p-4} \na v\cdot (D^2 v \cdot \na v)
	+ |\na v|^{p-2} \Del v
  \eas
  by (\ref{1.4}), from (\ref{1.3}) and (\ref{1.5}) we obtain that
  \bea{1.6}
	& & \hs{-12mm}
	\frac{1}{p} \frac{d}{dt} \io |\na v|^p 
	+ \frac{p-2}{4} \io \frac{1}{h(u,u_t)} |\na v|^{p-4} \Big| \na |\na v|^2 \Big|^2
	+ \io \frac{1}{h(u,u_t)} |\na v|^{p-2} |D^2 v|^2 \nn\\
	&=& - \frac{1}{2} \io |\na v|^{p-2} \cdot 2(D^2 v\cdot\na v) \cdot 
		\Big\{ - \frac{h_u(u,u_t)}{h^2(u,u_t)} \na u - \frac{h_{u_t}(u,u_t)}{h^2(u,u_t)} \na v 
		+ a\frac{h_{u_t}(u,u_t)}{h^2(u,u_t)} \na u \Big\} \nn\\
	& & + \io |\na v|^{p-2} \cdot \bigg\{ \na v\cdot 
		\Big\{ - \frac{h_u(u,u_t)}{h^2(u,u_t)} \na u - \frac{h_{u_t}(u,u_t)}{h^2(u,u_t)} \na v 
		+ a\frac{h_{u_t}(u,u_t)}{h^2(u,u_t)} \na u \Big\} \bigg\} \Del v \nn\\
	& & + a \io |\na v|^p
	- a^2 \io |\na v|^{p-2} \na u\cdot\na v \nn\\
	& & - \io \frac{f(u,u_t,\na u,\na u_t)}{h(u,u_t)} 
	|\na v|^{p-4} \cdot 
		\Big\{ (p-2)\na v\cdot (D^2 v\cdot\na v) + |\na v|^2 \Del v\Big\} \nn\\
	& & + \int_{\pO} \frac{1}{h(u,u_t)} |\na v|^{p-2} \cdot 
		\Big\{ f(u,u_t,\na u,\na u_t) \frac{\pa v}{\pa\nu} + \frac{1}{2} \frac{\pa |\na v|^2}{\pa\nu}\Big\}
	\qquad \mbox{for all } t\in (0,T).
  \eea
  Using that on $\pO\times (0,\tm)$ we have $u=u_t=0$ and hence also $v=0$, and that therefore (\ref{0v}) implies that
  \bas
	0= \Del v + f(u,u_t,\na u,\na u_t)
	\qquad \mbox{on } \pO\times (0,T)
  \eas
  and thus
  \bas
	f(u,u_t,\na u,\na u_t) \frac{\pa v}{\pa\nu}
	+ \frac{1}{2} \frac{\pa |\na v|^2}{\pa\nu}
	&\le& f(u,u_t,\na u,\na u_t) \frac{\pa v}{\pa\nu} + \frac{\pa v}{\pa\nu} \Del v + \kappa \big|\frac{\pa v}{\pa\nu}\Big|^2
	= \kappa \Big|\frac{\pa v}{\pa\nu}\Big|^2 \\
	&\le& \kappa |\na v|^2
	\qquad \mbox{on } \pO\times (0,T)
  \eas
  by Lemma \ref{lem_black}, rearranging (\ref{1.6}) and using that $p\ge 2$ in dropping the second summand therein 
  leads to (\ref{1.2}).
\qed
Appropriately estimating its right-hand side expressions turns (\ref{1.2}) into an inequality which exclusively
contains first-order spatial derivatives in its forcing part.
\begin{lem}\label{lem111}
  Assume (\ref{h}), and let $p\ge 2$.
  Then there exist $\Gamma_6=\Gamma_6(p,a,h)>0$ and $\Gamma_7=\Gamma_7(p,a,h)>0$ such that if (\ref{init}) holds and
  $t\in (0,\tm]$ is such that (\ref{uut}) is satisfied with $\del_1=\del_1(h)$ as in Lemma \ref{lem11}, then for any choice of $r>2$
  such that $r\ge p$,
  \bea{111.1}
	\hs{-4mm}
	\frac{d}{dt} \io |\na v|^p 
	+ \Gamma_6 \io |\na v|^{p-2} |D^2 v|^2
	&\le& \Gamma_7 \cdot \bigg\{ \io |\na u|^r \bigg\}^\frac{p}{r}
	+ \Gamma_7 \cdot \bigg\{ \io |\na v|^r \bigg\}^\frac{p}{r}
	+ \Gamma_7 \io |\na v|^{p+2} \nn\\
	& & + \Gamma_7 \cdot \bigg\{ \io |\na u|^r \bigg\}^\frac{2}{r} \cdot 
		\bigg\{ \io |\na v|^\frac{pr}{r-2} \bigg\}^\frac{r-2}{r} \nn\\
	& & + \Gamma_7 \io |\na v|^{p-2} f^2(u,u_t,\na u,\na u_t)
	\quad \mbox{for all } t\in (0,T).
  \eea
\end{lem}
\proof
  Since $h\in C^1([-\del_1,\del_1]^2)$, both $c_1\equiv c_1(h):=\|h_u\|_{L^\infty((-\del_1,\del_1)^2)}$ and
  $c_2\equiv c_2(h):=\|h_{u_t}\|_{L^\infty((-\del_1,\del_1)^2)}$ are finite, and using (\ref{11.1}) along with (\ref{uut}) and the
  pointwise inequality $|\Del v|\le \sqrt{n}|D^2 v|$, on the right-hand side of (\ref{1.2}) we can thereby estimate
  \bea{111.2}
	& & \hs{-20mm}
	\io \frac{h_u(u,u_t)}{h^2(u,u_t)} |\na v|^{p-2} \cdot \Big\{ \na u\cdot (D^2 v\cdot\na v) - (\na u \cdot\na v) \Del v \Big\}
		\nn\\
	&\le& \frac{4c_1}{h_0^2} \io \Big\{ |\na u| \cdot |\na v|^{p-1} |D^2 v|
	+ |\na u|\cdot |\na v|^{p-1} |\Del v| \Big\} \nn\\
	&\le& \frac{4(1+\sqrt{n})c_1}{h_0^2} \io |\na u| \cdot |\na v|^{p-1} |D^2 v|
	\qquad \mbox{for all } t\in (0,T)
  \eea
  and, similarly,
  \bea{111.3}
	& & \hs{-10mm}
	\io \frac{h_{u_t}(u,u_t)}{h^2(u,u_t)} |\na v|^{p-2} \cdot \Big\{ \na v\cdot (D^2 v\cdot\na v) - a\na u\cdot (D^2 v\cdot\na v)
		- |\na v|^2 \Del v + a(\na u\cdot\na v) \Del v \Big\} \nn\\
	&\le& \frac{4(1+\sqrt{n}) c_2}{h_0^2} \io |\na v|^p |D^2 v|
	+ \frac{4a(1+\sqrt{n})c_2}{h_0^2} \io |\na u| \cdot |\na v|^{p-1} |D^2 v|
	\qquad \mbox{for all } t\in (0,T).
  \eea
  Likewise,
  \bea{111.4}
	& & \hs{-20mm}
	- \io \frac{f(u,u_t,\na u,\na u_t)}{h(u,u_t)} |\na v|^{p-4} \cdot \Big\{ (p-2) \na v\cdot (D^2 v\cdot\na v)
		+ |\na v|^2 \Del v \Big\} \nn\\
	&\le& \frac{2(p-2+\sqrt{n})}{h_0} \io |f(u,u_t,\na u,\na u_t)| \cdot |\na v|^{p-2} |D^2 v|
	\qquad \mbox{for all } t\in (0,T),
  \eea
  while by Young's inequality,
  \be{111.5}
	- a^2 \io |\na v|^{p-2} \na u\cdot\na v
	\le a^2 \io |\na u|^p
	+ a^2 \io |\na v|^p
	\qquad \mbox{for all } t\in (0,T).
  \ee
  To control the rightmost summand in (\ref{1.2}), we rely on continuity of the boundary trace embedding from $W^{1,1}(\Om)$
  into $L^1(\pO)$ (\cite{Alt}) in fixing $c_3>0$ such that
  \bas
	\int_{\pO} |\vp| \le c_3 \io |\na\vp| + c_3 \io |\vp|
	\qquad \mbox{for all } \vp\in W^{1,1}(\Om),
  \eas
  so that, again by (\ref{11.1}),
  \bea{111.6}
	\kappa \int_{\pO} \frac{1}{h(u,u_t)} |\na v|^p
	&\le& \frac{2\kappa}{h_0} \int_{\pO} |\na v|^p \nn\\
	&\le& \frac{2\kappa c_3}{h_0} \io \Big| \na |\na v|^p \Big|
	+ \frac{2\kappa c_3}{h_0} \io |\na v|^p \nn\\
	&=& \frac{2p\kappa c_3}{h_0} \io \Big| |\na v|^{p-2} D^2 v\cdot\na v\Big|
	+ \frac{2\kappa c_3}{h_0} \io |\na v|^p \nn\\
	&\le& \frac{2p\kappa c_3}{h_0} \io |\na v|^{p-1} |D^2 v|
	+ \frac{2\kappa c_3}{h_0} \io |\na v|^p 
	\qquad \mbox{for all } t\in (0,T).
  \eea
  As (\ref{11.1}) furthermore implies that
  \bas
	\io \frac{1}{h(u,u_t)} |\na v|^{p-1} |D^2 v|^2
	\ge \frac{1}{2h_0} \io |\na v|^{p-2} |D^2 v|^2
	\qquad \mbox{for all } t\in (0,T),
  \eas
  from (\ref{1.2}) and (\ref{111.2})-(\ref{111.6}) we thus infer that
  \bea{111.7}
	\frac{d}{dt} \io |\na v|^p
	+ c_4 \io |\na v|^{p-2} |D^2 v|^2
	&\le& c_5 \io |\na u|\cdot |\na v|^{p-1} |D^2 v|
	+ c_5 \io |\na v|^p |D^2 v| \nn\\
	& & +c_5 \io |\na v|^{p-1} |D^2 v|
	+ c_5 \io |f(u,u_t,\na u,\na u_t)| \cdot |\na v|^{p-2} |D^2 v| \nn\\
	& & + c_5 \io |\na u|^p 
	+ c_5 \io |\na v|^p
	\qquad \mbox{for all } t\in (0,T)
  \eea
  with $c_4\equiv c_4(p,h):=\frac{p}{2h_0}$ and 
  $c_5\equiv c_5(p,a,h):= p\cdot \Big\{ \frac{4(1+\sqrt{n})c_1}{h_0^2} + \frac{4(a+1)(1+\sqrt{n})c_2}{h_0^2}
  + \frac{2(p+1)\kappa c_3}{h_0} + \frac{2(p-2+\sqrt{n})}{h_0} + a^2\Big\}$.
  Here we invoke Young's inequality to see that
  \bas
	c_5 \io |\na u|\cdot |\na v|^{p-1} |D^2 v|
	\le \frac{c_4}{8} \io |\na v|^{p-2} |D^2 v|^2
	+ \frac{2c_5^2}{c_4} \io |\na u|^2 |\na v|^p
  \eas
  and
  \bas
	c_5 \io |\na v|^p |D^2 v|
	\le \frac{c_4}{8} \io |\na v|^{p-2} |D^2 v|^2
	+ \frac{2c_5^2}{c_4} \io |\na v|^{p+2}
  \eas
  and
  \bas
	c_5 \io |\na v|^{p-1} |D^2 v|
	\le \frac{c_4}{8} \io |\na v|^{p-2} |D^2 v|^2
	+ \frac{2c_5^2}{c_4} \io |\na v|^p
  \eas
  as well as
  \bas
	c_5 \io |f(u,u_t,\na u,\na u_t)| \cdot |\na v|^{p-2} |D^2 v|
	\le \frac{c_4}{8} \io |\na v|^{p-2} |D^2 v|^2
	+ \frac{2c_5^2}{c_4} \io |\na v|^{p-2} f^2(u,u_t,\na u,\na u_t)
  \eas
  for all $t\in (0,T)$, whence we obtain that
  \bas
	\frac{d}{dt} \io |\na v|^p
	+ \frac{c_4}{2} \io |\na v|^{p-2} |D^2 v|^2
	&\le& c_5 \io |\na u|^p
	+ \Big\{ c_5 + \frac{2c_5^2}{c_4}\Big\} \io |\na v|^p \nn\\
	& & + \frac{2c_5^2}{c_4} \io |\na v|^{p+2}
	+ \frac{2c_5^2}{c_4} \io |\na u|^2 |\na v|^p \nn\\
	& & + \frac{2c_5^2}{c_4} \io |\na v|^{p-2} f^2 (u,u_t,\na u,\na u_t)
	\qquad \mbox{for all } t\in (0,T).
  \eas
  The proof can thus be completed by an application of the H\"older inequality in estimating
  \bas
	\io |\na u|^2 |\na v|^p
	\le \bigg\{ \io |\na u|^r \bigg\}^\frac{2}{r} \cdot
		\bigg\{ \io |\na v|^\frac{pr}{r-2}\bigg\}^\frac{r-2}{r}
  \eas
  and
  \bas
	\io |\na u|^p + \io |\na v|^p
	\le |\Om|^\frac{r-p}{r} \cdot \bigg\{ \io | \na u|^r \bigg\}^\frac{p}{r}
	+ |\Om|^\frac{r-p}{r} \cdot \bigg\{ \io | \na v|^r \bigg\}^\frac{p}{r}
  \eas
  for $t\in (0,T)$.
\qed
To prepare a control of those contributions to (\ref{111.1}) which contain the non-diffusible quantity $u$, 
let us record the following observation which is rather immediately implied by the mere definition in (\ref{v}).
We underline that thanks to our overall assumption on positivity of $a$, the resulting differential inequality
(\ref{1444.1}) again contains an absorptive linear summand.
\begin{lem}\label{lem1444}
  Assume (\ref{init}), and let $r\ge 2$ and $p>1$. Then
  \be{1444.1}
	\frac{d}{dt} \bigg\{ \io |\na u|^r \bigg\}^\frac{p}{r}
	+ a \cdot \bigg\{ \io |\na u|^r \bigg\}^\frac{p}{r}
	\le a^{1-p} \cdot \bigg\{ \io |\na v|^r \bigg\}^\frac{p}{r}
	\qquad \mbox{for all } t\in (0,\tm).
  \ee
\end{lem}
\proof
  Using the identity $u_t=v-au$, we see that
  \bas
	\frac{d}{dt} \bigg\{ \io |\na u|^r\bigg\}^\frac{p}{r}
	&=& p\cdot\bigg\{ \io |\na u|^r\bigg\}^\frac{p-r}{r} \cdot \io |\na u|^{r-2} \na u\cdot\na u_t \nn\\
	&=& p\cdot \bigg\{ \io |\na u|^r\bigg\}^\frac{p-r}{r} \cdot \io |\na u|^{r-2} \na u\cdot\na v \\
	& & - pa \cdot \bigg\{ \io |\na u|^r\bigg\}^\frac{p}{r}
	\qquad \mbox{for all } t\in (0,\tm).
  \eas
  Since H\"older's and Young's inequalities imply that
  \bas
	& & \hs{-30mm}
	p\cdot \bigg\{ \io |\na u|^r\bigg\}^\frac{p-r}{r} \cdot \io |\na u|^{r-2} \na u\cdot\na v \\
	&\le& p\cdot \bigg\{ \io |\na u|^r \bigg\}^\frac{p-r}{r} \cdot 
		\bigg\{ \io |\na u|^r \bigg\}^\frac{r-1}{r} \cdot \bigg\{ \io |\na v|^r \bigg\}^\frac{1}{r} \\
	&=& pa \cdot \bigg\{ \io |\na u|^r \bigg\}^\frac{p-1}{r} \cdot 
		\Bigg\{ \frac{1}{a} \cdot \bigg\{ \io |\na v|^r \bigg\}^\frac{1}{r} \Bigg\} \\
	&\le& pa\cdot \Bigg\{ \frac{p-1}{p} \cdot \bigg\{ \bigg\{ \io |\na u|^r \bigg\}^\frac{p-1}{r} \bigg\}^\frac{p}{p-1}
		+ \frac{1}{p} \cdot \bigg\{ \frac{1}{a} \cdot \bigg\{ \io |\na v|^r \bigg\}^\frac{1}{r} \bigg\}^p \Bigg\} \\
	&=& (p-1)a \cdot\bigg\{ \io |\na u|^r \bigg\}^\frac{p}{r}
	+ a^{1-p} \cdot\bigg\{ \io |\na v|^r\bigg\}^\frac{p}{r}
	\qquad \mbox{for all } t\in (0,\tm),
  \eas
  this entails (\ref{1444.1}).
\qed
It remains to adequately exploit our hypothesis on genuinely superlinear behavior of $f$ near the origin,
as implied by (\ref{f0}), in order to expediently estimate the expressions in (\ref{14.1}) and (\ref{111.1}) containing $f$.
We first address the former of these:
\begin{lem}\label{lem133}
  Assume (\ref{h}), (\ref{f0}) and (\ref{f1}) with some $\lam>1$, let $p\ge 2$ be such that $p>n$, 
  and let $r\ge p$ be such that $r\ge\lam$.
  Then there exist $\Gamma_8=\Gamma_8(p,r,a,h,f)>0$ as well as
  $\mu_1=\mu_1(f)>1$ and $\mu_2=\mu_2(h,f)>1$
  such that if (\ref{init}) holds and (\ref{uut}) is valid with some $T\in (0,\tm]$, then
  \bea{133.1}
	& & \hs{-30mm}
	\bigg\{ \io |u|\cdot |f(u,u_t,\na u,\na u_t)| \bigg\}^\frac{p}{2}
	+ \bigg\{ \io |v|\cdot |f(u,u_t,\na u,\na u_t)| \bigg\}^\frac{p}{2} \nn\\
	&\le& \Gamma_8 \cdot \bigg\{ \io |\na u|^r \bigg\}^{\frac{p}{r}\cdot\mu_1}
	+ \Gamma_8 \cdot \bigg\{ \io |\na u|^r \bigg\}^{\frac{p}{r}\cdot\mu_2} \nn\\
	& & + \Gamma_8\cdot\bigg\{ \io |\na v|^p \bigg\}^{\mu_1}
	+ \Gamma_8 \cdot\bigg\{ \io |\na v|^p \bigg\}^{\mu_2}
	\qquad \mbox{for all } t\in (0,T).
  \eea
\end{lem}
\proof
  Since $r\ge p>n$, by continuity of the embeddings $W^{1,r}(\Om)\hra W^{1,p}(\Om) \hra L^\infty(\Om)$ we can find $c_1=c_1(p,r)>0$ 
  such that
  \be{133.2}
	\|\vp\|_{L^\infty(\Om)} \le c_1 \|\na\vp\|_{L^p(\Om)}
	\quad \mbox{and} \quad
	\|\vp\|_{L^\infty(\Om)} \le c_1 \|\na\vp\|_{L^r(\Om)}
	\qquad \mbox{for all } \vp\in W_0^{1,r}(\Om),
  \ee
  so that 
  \bea{133.3}
	& & \hs{-24mm}
	\io |u|\cdot |f(u,u_t,\na u,\na u_t)|
	+ \io |v|\cdot |f(u,u_t,\na u,\na u_t)| \nn\\
	&\le& c_1 \|\na u\|_{L^r(\Om)} \io |f(u,u_t,\na u,\na u_t)|
	+ c_1 \|\na v\|_{L^p(\Om)} \io |f(u,u_t,\na u,\na u_t)|
  \eea
  for all $t\in (0,T)$.
  Here, using (\ref{f1}), (\ref{v}), the H\"older inequality and our assumptions that $r\ge\lam$ and $p\ge \lam$,
  for all $t\in (0,T)$ we can estimate
  \bea{133.4}
	\int_{\{|\na u|+|\na u_t|>1\}} |f(u,u_t,\na u,\na u_t)|
	&\le& K_f \int_{\{|\na u|+|\na u_t|>1\}} \big(1+|\na u|+|\na u_t|\big)^\lam \nn\\
	&\le& 2^\lam K_f \int_{\{|\na u|+|\na u_t|>1\}} \big(|\na u|+|\na u_t|\big)^\lam \nn\\
	&\le& 2^{2\lam} (a+1)^\lam K_f \io |\na u|^\lam
	+ 2^{2\lam} K_f \io |\na v|^\lam \nn\\
	&\le& c_2 \cdot \bigg\{ \io |\na u|^r\bigg\}^\frac{\lam}{r} 
	+ c_2\cdot\bigg\{ \io |\na v|^p \bigg\}^\frac{\lam}{p}
  \eea
  with $c_2\equiv c_2(p,r,f):=2^{2\lam}(a+1)^\lam K_f \cdot (|\Om|^\frac{r-\lam}{r} + |\Om|^\frac{p-\lam}{p})$,
  because (\ref{uut}) and Lemma \ref{lem11} warrant that $|u|+|u_t|\le \del_1\le 1$ in $\Om\times (0,T)$.\abs
  In regions of small gradients, we draw on (\ref{f0}) to fix $\vt=\vt(h,f)\in (0,1)$ and $c_3=c_3(a,h,f)>0$ such that
  \bas
	\big| f(\xi,\wh{\sig}-a\xi,\Xi,\wh{\Sigma}-a\Xi)\big|
	&\le& c_3\cdot \big( |\xi|^{1+\vt} + |\wh{\sig}|^{1+\vt} + |\Xi|^{1+\vt} + |\wh{\Sigma}|^{1+\vt}\big) \nn\\
	& & \hs{20mm}
	\mbox{for all $(\xi,\wh{\sig},\Xi,\wh{\Sigma})\in \R^{2n+2}$ such that $|\Xi|+|\wh{\Sigma}-a\Xi| \le 1$}.
  \eas
  Accordingly, by (\ref{v}), (\ref{133.2}) and the H\"older inequality,
  \bas
	\int_{\{|\na u|+|\na u_t|\le 1\}} |f(u,u_t,\na u,\na u_t)|
	&\le& c_3 \io |u|^{1+\vt}
	+ c_3 \io |v|^{1+\vt}
	+ c_3 \io |\na u|^{1+\vt}
	+ c_3 \io |\na v|^{1+\vt} \nn\\
	&\le& c_3 c_1^{1+\vt} |\Om| \cdot \|\na u\|_{L^r(\Om)}^{1+\vt}
	+ c_3 c_1^{1+\vt} |\Om| \cdot \|\na v\|_{L^p(\Om)}^{1+\vt} \nn\\
	& & + c_3 |\Om|^\frac{r-1+\vt}{r} \|\na u\|_{L^r(\Om)}^{1+\vt}
	+ c_3 |\Om|^\frac{p-1+\vt}{p} \|\na v\|_{L^p(\Om)}^{1+\vt} \nn\\
	&\le& c_4 \|\na u\|_{L^r(\Om)}^{1+\vt}
	+ c_4 \|\na v\|_{L^p(\Om)}^{1+\vt}
	\qquad \mbox{for all } t\in (0,T).,
  \eas
  where $c_4\equiv c_4(p,r,a,h,f):=c_3 \cdot (c_1^{1+\vt} |\Om| + |\Om|^\frac{r-1+\vt}{r} + |\Om|^\frac{p-1+\vt}{p})$.
  Together with (\ref{133.4}) inserted into (\ref{133.3}), thanks to Young's inequality this shows that
  \bas
	& & \hs{-2mm}
	\io |u|\cdot |f(u,u_t,\na u,\na u_t)|
	+ \io |v|\cdot |f(u,u_t,\na u,\na u_t)| \nn\\
	&\le& c_1\cdot \big( \|\na u\|_{L^r(\Om)} + \|\na v\|_{L^p(\Om)} \big) \cdot
	c_2 \cdot \big( \|\na u\|_{L^r(\Om)}^\lam + \|\na v\|_{L^p(\Om)}^\lam \big) \\
	& & + c_1\cdot \big( \|\na u\|_{L^r(\Om)} + \|\na v\|_{L^p(\Om)} \big) \cdot
	c_4 \cdot \big( \|\na u\|_{L^r(\Om)}^{1+\vt} + \|\na v\|_{L^p(\Om)}^{1+\vt} \big) \\
	&\le& 3c_1 c_2 \|\na u\|_{L^r(\Om)}^{\lam+1} 
	+ 3c_1 c_2 \|\na v\|_{L^p(\Om)}^{\lam+1} \nn\\
	& & + 3c_1 c_4 \|\na u\|_{L^r(\Om)}^{2+\vt}
	+ 3c_1 c_4 \|\na v\|_{L^p(\Om)}^{2+\vt}
	\qquad \mbox{for all } t\in (0,T),
  \eas
  and hence implies (\ref{133.1}) if we let $\mu_1\equiv \mu_1(f):=\frac{\lam+1}{2}$ 
  and $\mu_2\equiv \mu_2(h,f):=\frac{2+\vt}{2}$.
\qed
Quite similar arguments apply to the corresponding integrals in (\ref{111.1}):
\begin{lem}\label{lem144}
  Assume (\ref{h}), (\ref{f0}) and (\ref{f1}) with some $\lam>1$, let $p\ge 2$ be such that $p>n$, 
  and let $r\ge p$ satisfy $r\ge 2\lam$.
  Then there exist $\Gamma_9=\Gamma_9(p,r,a,h,f)>0$ and
  $\mu_3=\mu_3(h,f)>1$ such that whenever (\ref{init}) holds and $T\in (0,\tm]$ is such that (\ref{uut}) is satisfied
  with $\del_1=\del_1(h)$ as in Lemma \ref{lem11}, it follows that
  \bea{144.1}
	\io |\na v|^{p-2} f^2(u,u_t,\na u,\na u_t)
	&\le& \Gamma_9 \cdot \bigg\{ \io |\na u|^r \bigg\}^\frac{p}{r}
	+ \Gamma_9 \cdot \bigg\{ \io |\na v|^r \bigg\}^\frac{p}{r} \nn\\
	& & + \Gamma_9 \io |\na v|^{p-2+2\lam}
	+ \Gamma_9 \|\na u\|_{L^r(\Om)}^{2\lam} \|\na v\|_{L^\frac{(p-2)r}{r-2\lam}(\Om)}^{p-2} \nn\\
	& & + \Gamma_9 \cdot \bigg\{ \io |\na u|^r \bigg\}^{\frac{p}{r}\cdot\mu_3}
	+ \Gamma_9 \cdot \bigg\{ \io |\na v|^p \bigg\}^{\mu_3}
  \eea
  for all $t\in (0,T)$.
\end{lem}
\proof
  Our assumptions that $f(0,0,0,0)=0$ and $\na f(0,0,0,0)=0$ imply that since $\na f$ is locally H\"older continuous in
  $\R^{2n+2}$, with some $\vt=\vt(h,f)\in (0,1)$ and some $c_1=c_1(h,f)>0$ we have
  \bas
	f^2(\xi,\sig,\Xi,\Sigma)
	&\le& c_1 \cdot \big( |\xi|^{2+\vt} + |\sig|^{2+\vt} + |\Xi|^{2+\vt} + |\Sigma|^{2+\vt}\big) \\
	& & \hs{20mm}
	\mbox{for all $(\xi,\sig,\Xi,\Sigma) \in \R^{2n+2}$ such that $|\xi|+|\sig|\le \del_1$ and $|\Xi|+|\Sigma| \le 1$.}
  \eas
  Assuming (\ref{init}) and (\ref{uut}) with some $T\in (0,\tm]$, since $u_t=v-au$ we thus obtain that
  \bas
  	& & f^2(u,u_t,\na u,\na u_t) \hs{180mm} \\
	&\le& c_1\cdot \Big\{ |u|^{2+\vt} + 2^{2+\vt} (|v|^{2+\vt} + a^{2+\vt} |u|^{2+\vt}\big)
	+ |\na u|^{2+\vt} + 2^{2+\vt} \big( |\na v|^{2+\vt} + a^{2+\vt} |\na u|^{2+\vt}\big) \Big\} \\
	&\le& c_2\cdot \big( |u|^{2+\vt} + |v|^{2+\vt} + |\na u|^{2+\vt} + |\na v|^{2+\vt}\big)
  \eas
  in $(\Om\times (0,T))\cap \{|\na u|+|\na u_t| \le 1\}$, where $c_2\equiv c_2(h,f,a):=c_1\cdot (1+2^{2+\vt})(1+c_1^{2+\vt})$.
  Since the inequality $r\ge p$ and the continuity of the embedding $W^{1,p}(\Om) \hra L^\frac{(2+\vt)p}{2}(\Om)$,
  as particularly implied by the inequality $p>n$, entails the existence of $c_3=c_3(p,h,f)>0$ and $c_4=c_4(p,h,f)>0$ such that
  \bas
	\|\vp\|_{L^\frac{(2+\vt)p}{2}(\Om)}^{2+\vt}
	\le c_3 \|\na\vp\|_{L^p(\Om)}^{2+\vt}
	\le c_4 \|\na\vp\|_{L^r(\Om)}^{2+\vt}
	\qquad \mbox{for all } \vp\in W_0^{1,r}(\Om),
  \eas
  by means of the H\"older inequality and the fact that $|\na v| = |\na (u_t+au)| \le 1+a$ in 
  $\{|\na u| + |\na u_{t}| \le 1\}$ we therefore obtain that 
  \bas
	& & \hs{-30mm}
	\int_{\{|\na u|+|\na u_t|\le 1 \}} |\na v|^{p-2} f^2(u,u_t,\na u,\na u_t) \\
	&\le& c_2 \io |u|^{2+\vt} |\na v|^{p-2}
	+ c_2 \io |v|^{2+\vt} |\na v|^{p-2} \\
	& & + c_2 \int_{\{ |\na u|+|\na u_t|\le 1\}} |\na u|^{2+\vt} |\na v|^{p-2}
	+ c_2 \int_{\{ |\na u|+ |\na u_t| \le 1\}} |\na v|^{p+\vt} \\
	&\le& c_2 \|u\|_{L^\frac{(2+\vt)p}{2}(\Om)}^{2+\vt} \|\na v\|_{L^p(\Om)}^{p-2}
	+ c_2 \|v\|_{L^\frac{(2+\vt)p}{2}(\Om)}^{2+\vt} \|\na v\|_{L^p(\Om)}^{p-2} \\
	& & + c_2 \io |\na u|^2 |\na v|^{p-2}
	+ c_2 \cdot (1+a)^{\vt} \io |\na v|^p \\
	&\le& c_2 c_4 \|\na u\|_{L^r(\Om)}^{2+\vt} \|\na v\|_{L^p(\Om)}^{p-2}
	+ c_2 c_3 \|\na v\|_{L^p(\Om)}^{p+\vt} \\
	& & + c_2 \|\na u\|_{L^p(\Om)}^2 \|\na v\|_{L^p(\Om)}^{p-2}
	+ c_2 (1+a)^{\vt} \io |\na v|^p
  \eas
  for all $t\in (0,T)$.
  Since with $c_5\equiv c_5(p,r):=|\Om|^\frac{r-p}{pr}$ we have $\|\vp\|_{L^p(\Om)} \le c_5 \|\vp\|_{L^r(\Om)}$ for all 
  $\vp\in L^r(\Om)$ by the H\"older inequality, and since Young's inequality together with the hypothesis $p\ge 2$ guarantees that
  \bas
	\|\na v\|_{L^p(\Om)}^{p+\vt}
	= \bigg\{ \io |\na v|^p\bigg\}^\frac{p+\vt}{p}
	\le \bigg\{ \io |\na v|^p \bigg\}^\frac{2+\vt}{2}
	+ \io |\na v|^p
	\qquad \mbox{for all } t\in (0,T)
  \eas
  due to the fact that $1\le \frac{p+\vt}{p} \le \frac{2+\vt}{2}$, according to Young's inequality this entails that
  for all $t\in (0,T)$,
  \bea{144.2}
	\int_{\{|\na u|+|\na u_{t}|\le 1\}} |\na v|^{p-2} f^2(u,u_t,\na u,\na u_t)
	&\le& c_2 c_4 c_5^p \cdot \bigg\{ \io |\na v|^r \bigg\}^\frac{p}{r}
	+ c_2 c_4 \cdot \bigg\{ \io |\na u|^r \bigg\}^{\frac{p}{r} \cdot \frac{2+\vt}{2}} \nn\\
	& & + c_2 c_3\cdot \bigg\{ \io |\na v|^p \bigg\}^\frac{2+\vt}{2}
	+ c_2 c_3 c_5^p \cdot \bigg\{ \io |\na v|^r \bigg\}^\frac{p}{r} \nn\\
	& & + c_2 c^p_5 \cdot \bigg\{ \io |\na u|^r \bigg\}^\frac{p}{r}
	+ c_2 c_5^p \cdot \bigg\{ \io |\na v|^r \bigg\}^\frac{p}{r} \nn\\
	& & + c_2\cdot (1+a)^{2+\vt} c_5^p \cdot \bigg\{ \io |\na v|^r \bigg\}^\frac{p}{r}.
  \eea
  In the corresponding complementary region, we may rely on (\ref{f1}) to see that, again due to (\ref{v}),
  \bas
	f^2(u,u_t,\na u, \na u_t)
	&\le& K_f^2 \cdot \big( 1+ |\na u| + |\na u_t|\big)^{2\lam} \\
	&\le& (2^\lam K_f)^2 \cdot \big( |\na u| + |\na u_t|\big)^{2\lam} \\
	&\le& (2^\lam K_f^2) \cdot \big( (1+a) |\na u| + |\na v|\big)^{2\lam} \\
	&\le& c_6 |\na u|^{2\lam} + c_6 |\na v|^{2\lam}
	\qquad \mbox{in } (\Om\times (0,T))\cap \{ |\na u| + |\na u_{t}|>1\}
  \eas
  with $c_6\equiv c_6(a,f):=(4^\lam K_f)^2 (1+a)^{2\lam}$, because (\ref{uut}) along with the restriction $\del_1$ 
  in Lemma \ref{lem11} ensures that $|u|+|u_t|\le 1$ in $\Om\times (0,T)$.
  Consequently, the assumption $r\ge 2\lam$ guarantees that 
  due to the H\"older inequality,  
  for all $t\in (0,T)$ we have
  \bas
	\int_{\{|\na u|+|\na u_t|>1\}} |\na v|^{p-2} f^2(u,u_t,\na u,\na u_t)
	&\le& c_6 \io |\na u|^{2\lam} |\na v|^{p-2}
	+ c_6 \io |\na v|^{p-2+2\lam} \\
	&\le& c_6 \|\na u\|_{L^r(\Om)}^{2\lam} \|\na v\|_{L^\frac{(p-2)r}{r-2\lam}(\Om)}^{p-2}
	+ c_6 \io |\na v|^{p-2+2\lam},
  \eas
  which combined with (\ref{144.2}) establishes (\ref{144.1}) with $\mu_3\equiv \mu_3(h,f):=\frac{2+\vt}{2}$.
\qed
\mysection{Closing the loop. Proof of Theorem \ref{theo18}}\label{sect6}
Having the outcomes of the previous two sections at hand, we are now prepared to derive an inequality of the form announced 
in (\ref{010}) for the combined energy functional in (\ref{09}), yet conditional in the sense that the smallness assumption in
(\ref{uut}) is presupposed.
\begin{lem}\label{lem15}
  Assume (\ref{h}), (\ref{f0}) and (\ref{f1}) 
  with some $\lam>1$,
  and let $p\ge 2$ and $r\ge \max\{p,2\lam\}$ be such that $p>n$ and
  \be{15.1}
	p\ge n(\lam-1)
  \ee
  as well as
  \be{15.01}
	r>2\lam
	\qquad \mbox{if } n=2
  \ee
  and
  \be{15.02}
	\frac{np\lam}{n+p-2} \le r < \frac{np}{n-2}
	\qquad \mbox{if } n\ge 3.
  \ee
  Then there exist $b_1=b_1(p,r,a,h,f)>0, b_2=b_2(p,r,a,h,f)>0$, $\kappa_1=\kappa_1(p,f)>1$ and $\kappa_2=\kappa_2(p,h,f)>1$
  as well as $\Gamma_{10}=\Gamma_{10}(p,r,a,h,f)>0$
  and $\Gamma_{11}=\Gamma_{11}(p,r,a,h,f)>0$ such that whenever (\ref{init}) and (\ref{uut}) hold with some $T\in (0,\tm]$,
  the function defined by
  \be{F}
	\F(t)
	:= \io |\na v|^p
	+ b_1 \cdot \bigg\{ \io |\na u|^r \bigg\}^\frac{p}{r}
	+ b_2 \F_1^\frac{p}{2}(t)
	\qquad t\in [0,\tm),
  \ee
  with $\F_1$ taken from (\ref{F1}) satisfies
  \be{15.2}
	\F'(t) + \Gamma_{10} \F(t)
	+ \Big\{ \Gamma_{10} - \Gamma_{11} \F^{\kappa_1}(t) \Big\} \cdot \io |\na v|^{p-2} |D^2 v|^2
	\le \Gamma_{11} \F^{\kappa_2}(t)
	\qquad \mbox{for all } t\in (0,T).
  \ee
\end{lem}
\proof
  We let $\Gamma_i=\Gamma_i(p,a,h,f)>0$, 
  $i\in\{4,5,6,7\}$,
  be as in Lemma \ref{lem111} and Lemma \ref{lem144}, and take $\Gamma_i=\Gamma_i(p,r,a,h,f)>0$,
  $i\in\{8,9\}$,
  as well as $\mu_i=\mu_i(p,r,f)>1$, $i\in\{1,2,3\}$, from 
  Lemma \ref{lem144} and Lemma \ref{lem133}.
  We then fix $b_1=b_1(p,r,a,h,f)>0$ large enough such that
  \be{15.3}
	\frac{b_1 a}{2} \ge \Gamma_7 (1+\Gamma_9),
  \ee
  and apply 
  Corollary \ref{cor312}
  to 
  \be{15.33}
	\eps:=\frac{\Gamma_6}{2\cdot\big\{ \Gamma_7(1+\Gamma_9) + b_1 a^{1-p}\big\}},
  \ee
  to obtain $\Gamma_1=\Gamma_1(\eps,p,r)$ with the properties specified there.
  We thereupon pick $b_2=b_2(p,r,a,h)>0$ large enough fulfilling
  \be{15.4}
	\frac{b_2 \Gamma_4  \Gamma^{\frac{p}{2}}_2}{2} \ge \Gamma_1\cdot\big\{ \Gamma_7(1+\Gamma_9) + b_1 a^{1-p}\big\},
  \ee
  and assuming $u_0, u_{0t}$ and $T\in (0,\tm]$ to be such that (\ref{init}) and (\ref{uut}) hold, we let $\F$ be as defined
  in (\ref{F}).
  A combination of Lemma \ref{lem111} with Lemma \ref{lem144} then shows that since $r\ge 2\lam$,
  \bas
	\frac{d}{dt} \io |\na v|^p
	+ \Gamma_6 \io |\na v|^{p-2} |D^2 v|^2
	&\le& \Gamma_7(1+\Gamma_9) \cdot\bigg\{ \io |\na u|^r \bigg\}^\frac{p}{r}
	+ \Gamma_7(1+\Gamma_9) \cdot\bigg\{ \io |\na v|^r\bigg\}^\frac{p}{r} \\
	& & + \Gamma_7 \cdot \bigg\{ \io |\na u|^r \bigg\}^\frac{2}{r} \cdot \bigg\{ \io |\na v|^\frac{pr}{r-2}\bigg\}^\frac{r-2}{r}
	+ \Gamma_7 \io |\na v|^{p+2} \\
	& & + \Gamma_7 \Gamma_9 \io |\na v|^{p-2+2\lam}
	+ \Gamma_7 \Gamma_9 \|\na u\|_{L^r(\Om)}^{2\lam} \|\na v\|_{L^\frac{(p-2)r}{r-2\lam}(\Om)}^{p-2} \\
	& & + \Gamma_7 \Gamma_9 \cdot\bigg\{ \io |\na u|^r \bigg\}^{\frac{p}{r}\cdot\mu_3}
	+ \Gamma_7 \Gamma_9 \cdot \bigg\{ \io |\na v|^p \bigg\}^{\mu_3}
  \eas
  for all $t\in (0,T)$, while Corollary \ref{cor14} together with Lemma \ref{lem13} implies that
  \bas
	& & \hs{-20mm}
	\frac{d}{dt} \F_1^\frac{p}{2}(t)
	+  \frac{ \Gamma_4}{2}\F_1^\frac{p}{2}(t) +  \frac{ \Gamma_4  \Gamma^{\frac{p}{2}}_2}{2} \cdot \bigg\{ \io v^2\bigg\}^\frac{p}{2} \nn\\
	&\le& \Gamma_5 \Gamma_8 \cdot \bigg\{ \io |\na u|^r \bigg\}^{\frac{p}{r}\cdot \mu_1}
	+ \Gamma_5 \Gamma_8 \cdot \bigg\{ \io |\na u|^r \bigg\}^{\frac{p}{r}\cdot \mu_2} \nn\\
	& & + \Gamma_5 \Gamma_8 \cdot \bigg\{ \io |\na v|^p \bigg\}^{\mu_1}
	+ \Gamma_5 \Gamma_8 \cdot \bigg\{ \io |\na v|^p \bigg\}^{\mu_2} \nn\\
	& & +\Gamma_5\cdot \bigg\{ \io |u|^3 \bigg\}^\frac{p}{2}
	+ \Gamma_5 \cdot \bigg\{ \io |v|^3 \bigg\}^\frac{p}{2}
  \eas
  for all $t\in (0,T)$.
  In conjunction with Lemma \ref{lem1444}, this entails that if for $t\in (0,T)$ we abbreviate 
  $I(t):=\Gamma_6 \io |\na v|^{p-2} |D^2 v|^2$, then
  \bea{15.5}
	& & \hs{-16mm}
	\F'(t) 
	+ I(t)
	+ b_1 a\cdot \bigg\{ \io |\na u|^r \bigg\}^\frac{p}{r}
	+ \frac{ b_2 \Gamma_4}{2}
	\F_1^\frac{p}{2}(t)
	+ \frac{b_2 \Gamma_4  \Gamma^{\frac{p}{2}}_2}{2}
	\cdot \bigg\{ \io v^2 \bigg\}^\frac{p}{2} \nn\\
	&\le& \Gamma_7 (1+\Gamma_9) \cdot\bigg\{ \io |\na u|^r \bigg\}^\frac{p}{r}
	+ \big\{ \Gamma_7(1+\Gamma_9) + b_1 a^{1-p} \big\}\cdot \bigg\{ \io |\na v|^r \bigg\}^\frac{p}{r} \nn\\
	& & + \Gamma_7 \cdot \bigg\{ \io |\na u|^r \bigg\}^\frac{2}{r} 
		\cdot \bigg\{ \io |\na v|^\frac{pr}{r-2} \bigg\}^\frac{r-2}{r}
	+ \Gamma_7 \Gamma_9 \io |\na v|^{p-2+2\lam} 
	+ \Gamma_7\io |\na v|^{p+2}
	\nn\\
	& & + \Gamma_7 \Gamma_9 \|\na u\|_{L^r(\Om)}^{2\lam} \|\na v\|_{L^\frac{(p-2)r}{r-2\lam}(\Om)}^{p-2} 
	+ \Gamma_7 \Gamma_9 \cdot \bigg\{ \io |\na u|^r \bigg\}^{\frac{p}{r}\cdot\mu_3}
	+ \Gamma_7 \Gamma_9 \cdot\bigg\{ \io |\na v|^p\bigg\}^{\mu_3} \nn\\
	& & + b_2 \Gamma_5 \Gamma_8 \cdot \bigg\{ \io |\na u|^r\bigg\}^{\frac{p}{r}\cdot\mu_1}
	+ b_2\Gamma_5 \Gamma_8\cdot\bigg\{ \io |\na u|^r \bigg\}^{\frac{p}{r}\cdot\mu_2} \nn\\
	& & + b_2 \Gamma_5 \Gamma_8\cdot\bigg\{ \io |\na v|^p\bigg\}^{\mu_1}
	+ b_2\Gamma_5 \Gamma_8\cdot\bigg\{ \io |\na v|^p\bigg\}^{\mu_2} \nn\\
	& & + b_2 \Gamma_5 \cdot \bigg\{ \io |u|^3 \bigg\}^\frac{p}{2}
	+ b_2 \Gamma_5 \cdot\bigg\{ \io |v|^3 \bigg\}^\frac{p}{2}
	\qquad \mbox{for all } t\in (0,T).
  \eea
  Here, 
  \be{15.6}
	\Gamma_7(1+\Gamma_9)\cdot\bigg\{ \io |\na u|^r\bigg\}^\frac{p}{r}
	\le \frac{b_1 a}{2} \cdot\bigg\{ \io |\na u|^r \bigg\}^\frac{p}{r}
	\qquad \mbox{for all } t\in (0,T)
  \ee
  by (\ref{15.3}), whereas in line with our definition of $\eps$ in (\ref{15.33}) we infer from Corollary \ref{cor312} and
  (\ref{15.4}) that
  \bea{15.7}
	\hs{-6mm}
	\big\{ \Gamma_7(1+\Gamma_9) + b_1 a^{1-p}\big\} \cdot\bigg\{ \io|\na v|^r\bigg\}^\frac{p}{r} 
	&\le& \frac{1}{2} I(t)
	+ \Gamma_1\cdot\big\{ \Gamma_7(1+\Gamma_9) + b_1 a^{1-p}\big\} \cdot\bigg\{ \io v^2\bigg\}^\frac{p}{2} \nn\\
	&\le& \frac{1}{2} I(t)
	+ \frac{b_2 \Gamma_4  \Gamma^{\frac{p}{2}}_2}{2}
	\cdot\bigg\{ \io v^2\bigg\}^\frac{p}{2} 
	\qquad \mbox{for all } t\in (0,T).
  \eea
  Next, since the inequality $p\ge n$ ensures that $(p+2)\le \frac{(n+2)p}{n}$, we may draw on Lemma \ref{lem3}
  in choosing $c_1=c_1(p,a,h,f)>0$ fulfilling
  \be{15.77}
	\Gamma_7 \io |\na v|^{p+2}
	\le c_1 \cdot \bigg\{ \io |\na v|^p \bigg\}^\frac{2}{p} \cdot \io |\na v|^{p-2} |D^2 v|^2
	\le \frac{c_1}{\Gamma_6} \F^\frac{2}{p}(t) I(t)
	\qquad \mbox{for all } t\in (0,T).
  \ee
  Likewise,
  since $r\ge 2\lam>2$ and $r\ge p>n$, it follows that $\frac{pr}{r-2} < \frac{np}{(n-2)_+}$, whence Lemma \ref{lem2}
  applies so as to provide $c_2=c_2(p,r,a,h,f)>0$ such that
  \be{15.8}
	\bigg\{ \io |\na v|^\frac{pr}{r-2}\bigg\}^\frac{r-2}{r}
	\le c_2 I(t)
	\qquad \mbox{for all } t\in (0,T),
  \ee
  while similarly the inequalities $r\ge \frac{np\lam}{p+n-2}$ when $n\ge 3$ and $r>2\lam$ when $n=2$ 
  ensure that $\frac{(p-2)r}{r-2\lam}$ is a finite number not greater than $\frac{np}{(n-2)_+}$,
  and that thus from Lemma \ref{lem2} we obtain $c_3=c_3(p,r,a,h,f)>0$ satisfying 
  \be{15.9}
	\|\na v\|_{L^\frac{(p-2)r}{r-2\lam}(\Om)}^{p-2}
	\le c_3 I^\frac{p-2}{p}(t)
	\qquad \mbox{for all } t\in (0,T).
  \ee
  Apart from that, our assumption that $p\ge n(\lam-1)$ implies that $p-2+2\lam\le\frac{(n+2)p}{n}$, so that Lemma \ref{lem3} 
  asserts the existence of $c_4=c_4(p,a,h,f)>0$ such that
  \be{15.10}
	\io |\na v|^{p-2+2\lam}
	\le c_4\cdot\bigg\{ \io |\na v|^p\bigg\}^\frac{2(\lam-1)}{p} \cdot I(t)
	\qquad \mbox{for all } t\in (0,T),
  \ee
  and we finally make use of the continuity of the embeddings $W^{1,r}(\Om) \hra L^3(\Om)$ and $W^{1,p}(\Om) \hra L^3(\Om)$,
  as guaranteed by 
  the inequality $p>n$
  and the fact that $r\ge p$,
  to pick $c_5=c_5(p,r)>0$ and $c_6=c_6(p)>0$ fulfilling
  \be{15.11}
	\bigg\{ \io |\vp|^3 \bigg\}^\frac{p}{2} \le c_5\cdot\bigg\{ \io |\na\vp|^r\bigg\}^\frac{3p}{2r}
	\quad \mbox{and} \quad
	\bigg\{ \io |\vp|^3 \bigg\}^\frac{p}{2} \le c_6\cdot\bigg\{ \io |\na\vp|^p\bigg\}^\frac{3}{2}
	\qquad \mbox{for all } \vp\in W_0^{1,r}(\Om).
  \ee
  Using (\ref{F}) to estimate
  \bas
	\bigg\{ \io |\na u|^r\bigg\}^\frac{p}{r} \le \frac{1}{b_1} \cdot\F(t)
	\quad \mbox{and} \quad
	\io |\na v|^p \le \F(t)
	\qquad \mbox{for all } t\in (0,T),
  \eas
  and noting that another application of Lemma \ref{lem2} yields $c_7=c_7(p)>0$ such that
  \bas
	\frac{1}{4} I(t) \ge c_7 \io |\na v|^p
	\qquad \mbox{for all } t\in (0,T),
  \eas
  from (\ref{15.5})-(\ref{15.11}) we altogether infer that
  \bas
	& & \hs{-20mm}
	\F'(t) + \frac{1}{4} I(t) + c_7 \io |\na v|^p
	+ \frac{b_1 a}{2} \cdot\bigg\{ \io |\na u|^r\bigg\}^\frac{p}{r}
	+ b_2 \Gamma_4 \cdot \F_1^\frac{p}{2}(t) \nn\\
	&\le& \Gamma_7\cdot\Big( \frac{1}{b_1 } \F(t)\Big)^\frac{2}{p} \cdot c_1 I(t) 
	+ \frac{c_1}{\Gamma_6} \F^\frac{2}{p}(t) I(t)
	\nn\\
	& & + \Gamma_7 \Gamma_9 \cdot c_4 \F^\frac{2(\lam-1)}{p}(t) I(t)
	+ \Gamma_7 \Gamma_9 \cdot\Big(\frac{1}{b_1} \F(t)\Big)^\frac{2\lam}{p} \cdot c_3 I^\frac{p-2}{p}(t) \nn\\
	& & + \Gamma_7\Gamma_9 \cdot\Big(\frac{1}{b_1 }\F(t)\Big)^{\mu_3}
	+ \Gamma_7 \Gamma_9 \cdot \F^{\mu_3}(t) \nn\\
	& & + b_2 \Gamma_5 \Gamma_8\cdot\Big(\frac{1}{b_1 }\F(t)\Big)^{\mu_1}
	+ b_2 \Gamma_5 \Gamma_8 \cdot\Big(\frac{1}{b_1 }\F(t)\Big)^{\mu_2} \nn\\
	& & + b_2 \Gamma_5 \Gamma_8\cdot\F^{\mu_1}(t)
	+ b_2\Gamma_5 \Gamma_8 \cdot \F^{\mu_2}(t) \nn\\
	& & + b_2 \Gamma_5 c_5\cdot\Big(\frac{1}{b_1 }\F(t)\Big)^\frac{3}{2} 
	+ b_2 \Gamma_5 c_6\cdot\F^\frac{3}{2}(t)
	\qquad \mbox{for all } t\in (0,T).
  \eas
  On abbreviating $c_8\equiv c_8(p,r,a,h,f):=\min\{ c_7,\frac{a}{2},\Gamma_4\}$ and
  $c_9\equiv c_9(p,r,a,h,f):=\Gamma_7 \cdot \big( \frac{1}{b_1 }\big)^\frac{2}{p} c_2
  + \frac{c_1}{\Gamma_6}
  + \Gamma_7 \Gamma_9 \cdot\big\{ c_4 + \big(\frac{1}{b_1 }\big)^\frac{2\lam}{p} c_3 + \big(\frac{1}{b_1 }\big)^{\mu_3} 
  +1 \big\}
  + b_2 \Gamma_5 \Gamma_8 \cdot\big\{ \big(\frac{1}{b_1 }\big)^{\mu_1} + \big(\frac{1}{b_1 }\big)^{\mu_2} +2\big\}
  + b_2 \Gamma_5 \cdot \big\{ c_5\cdot \big(\frac{1}{b_1 }\big)^\frac{3}{2} + c_6\big\}$,
  in view of (\ref{F}) we can turn this into the inequality
  \bea{15.111}
	\hs{-8mm}
	\F'(t) + c_8 \F(t) + \frac{1}{4} I(t)
	&\le& c_9 \F^\frac{2}{p}(t) I(t)
	+ c_9 \F^\frac{2(\lam-1)}{p}(t) I(t)
	+ c_9 \F^\frac{2\lam}{p}(t) I^\frac{p-2}{p}(t) \nn\\
	& & + c_9 \F^{\mu_1}(t)
	+ c_9 \F^{\mu_2}(t) 
	+ c_9 \F^{\mu_3}(t)
	+ c_9 \F^\frac{3}{2}(t)
	\qquad \mbox{for all } t\in (0,T).
  \eea
  Here, several applications of Young's inequality show that
  \bas
	c_9 \F^\frac{2\lam}{p}(t) I^\frac{p-2}{p}(t)
	&=& \Big(\frac{1}{16} I(t)\Big)^\frac{p-2}{p} \cdot 16^\frac{p-2}{p} c_9 \F^\frac{2\lam}{p}(t) \nn\\
	&\le& \frac{1}{16} I(t)
	+ 16^\frac{p-2}{2} c_9^\frac{p}{2} \F^\lam(t)
	\qquad \mbox{for all } t\in (0,T),
  \eas
  and that if we let $\kappa_1\equiv \kappa_1(p,f):=\max\{\frac{2}{p}, \frac{2(\lam-1)}{p}\}$,
  $\kappa_2\equiv \kappa_2(p,h,f):=\max\{\mu_1,\mu_2,\mu_3,\frac{3}{2},\lam\}$ and
  $c_{10}\equiv c_{10}(p,r,a,h,f):=\max\{c_9,16^\frac{p-2}{2} c_9^\frac{p}{2}\}$, then
  \bas
	c_8 \F^\kappa(t)
	&=& \Big(\frac{1}{16}\Big)^\frac{\kappa_1-\kappa}{\kappa_1} \cdot \Big\{ 16^\frac{\kappa_1-\kappa}{\kappa_1} c_9 \F^\kappa(t)
		\Big\} \\
	&\le& \frac{1}{16} + 16^\frac{\kappa_1-\kappa}{\kappa} c_9^\frac{\kappa_1}{\kappa} \F^{\kappa_1}(t)
	\qquad \mbox{for all $t\in (0,T)$ and any } \kappa\in (0,\kappa_1]
  \eas
  as well as
  \bas
	c_{10} \F^\kappa(t)
	&=& \Big(\frac{c_8}{10} \F(t)\Big)^\frac{\kappa_2-\kappa}{\kappa_2-1} \cdot
		\Big(\frac{10}{c_8}\Big)^\frac{\kappa_2-\kappa}{\kappa_2-1} c_{10} \F^\frac{\kappa_2(\kappa-1)}{\kappa_2-1}(t) \\
	&\le& \frac{c_8}{10} \F(t)
	+ \Big(\frac{10}{c_8}\Big)^\frac{\kappa_2-\kappa}{\kappa-1}
		c_{10}^\frac{\kappa_2-1}{\kappa-1} \F^{\kappa_2}(t)
	\qquad \mbox{for all $t\in (0,T)$ and each } \kappa\in (0,\kappa_2].
  \eas
  Writing $c_{11}\equiv c_{11}(p,r,a,h,f):=\max\Big\{ 16^{(\kappa_1-\frac{2}{p})/(\frac{2}{p})} c_9^\frac{p\kappa_1}{2},
  16^{(\kappa_1-\frac{2(\lam-1)}{p})/(\frac{2(\lam-1)}{p})} c_9^\frac{p\kappa_1}{2(\lam-1)} \Big\}$
  as well as
  $c_{12}\equiv c_{12}(p,r,a,h,f):=\max\Big\{ (\frac{10}{c_8})^\frac{\kappa_2-\kappa}{\kappa-1} 
  c_{10}^\frac{\kappa_2-1}{\kappa-1} \ \Big|
  \ \kappa\in \{\mu_1,\mu_2,\mu_3,\frac{3}{2},\lam\}\Big\}$, 
  from (\ref{15.111}) we therefore obtain that
  for all $t\in (0,T)$,
  \bas
	\F'(t) + c_8\F(t) + \frac{1}{4} I(t)
	&\le& 2\cdot \Big\{ \frac{1}{16} + c_{11} \F^{\kappa_1}(t) \Big\} \cdot I(t)
	+ \frac{1}{16} I(t)
	+ 5\cdot\Big\{ \frac{c_8}{10} \F(t) + c_{12} \F^{\kappa_2}(t)\Big\},
  \eas
  that is,
  \bas
	\F'(t) + \frac{c_8}{2} \F(t)
	+ \Big\{ \frac{1}{16} - 2c_{11} \F^{\kappa_1}(t)\Big\} \cdot I(t)
	\le 5 c_{12} \F^{\kappa_2}(t)
	\qquad \mbox{for all } t\in (0,T),
  \eas
  which yields the claim.
\qed
Under smallness assumptions on the initial data in the style of those from Theorem \ref{theo18}, by means of an integration
in (\ref{15.2}) we can make sure that (\ref{uut}) in fact holds throughout evolution, hence closing our loop 
of arguments in the following sense.
\begin{lem}\label{lem16}
  Assume (\ref{h}), (\ref{f0}) and (\ref{f1}) 
  with some $\lam>1$ and $K_f>0$, and suppose that $p\ge 2$ and $r\ge \max\{p,2\lam\}$ are such that $p>n$, and that
  (\ref{15.1}), (\ref{15.01}) and (\ref{15.02}) hold.
  Then there exist $\eta=\eta(p,r,a,h,f)>0$, $\beta=\beta(p,r,a,h,f)>0$
  and $\Gamma_{12}=\Gamma_{12}(p,r,a,h,f)>0$ such that if (\ref{init}) is valid with
  \be{16.1}
	\io |\na u_0|^r \le \eta
	\qquad \mbox{and} \qquad
	\io |\na u_{0t}|^p \le \eta,
  \ee
  then
  \be{16.2}
	|u| + |u_t| \le \del_1
	\qquad \mbox{in } \Om\times (0,\tm)
  \ee
  and
  \be{16.3}
	\io |\na u(\cdot,t)|^r
	+ \io |\na u_t(\cdot,t)|^p
	\le \Gamma_{12} e^{-\beta t}
	\qquad \mbox{for all } t\in (0,\tm),
  \ee
  where $\del_1=\del_1(h)$ is as in Lemma \ref{lem11}.
\end{lem}
\proof
  We let $b_1=b_1(p,r,a,h,f), b_2=b_2(p,r,a,h,f), \kappa_1=\kappa_1(p,f), \kappa_2=\kappa_2(p,h,f)>1, 
  \Gamma_{10}=\Gamma_{10}(p,r,a,h,f)$ and $\Gamma_{11}=\Gamma_{11}(p,r,a,h,f)$ 
  be as in Lemma \ref{lem15}, and take $\Gamma_3=\Gamma_3(a,h)$ from Lemma \ref{lem13}.
  Using that $\min\{p,r\}>n$, we then choose $c_1=c_1(r)>0$ and $c_2=c_2(p)>0$ such that
  \be{16.4}
	\|\vp\|_{L^\infty(\Om)} \le c_1 \|\na\vp\|_{L^r(\Om)}
	\quad \mbox{and} \quad
	\|\vp\|_{L^\infty(\Om)} \le c_2 \|\na\vp\|_{L^p(\Om)}
	\quad \mbox{for all } \vp \in W_0^{1,r}(\Om),
  \ee
  and pick $\eta_1=\eta_1(p,r,a,h,f)>0$ small enough fulfilling
  \be{16.5}
	(1+a) c_1 \cdot\Big(\frac{\eta_1}{b_1}\Big)^\frac{1}{p} + c_2\eta_1^\frac{1}{p} \le \del_1
  \ee
  and
  \be{16.6}
	\Gamma_{10} - \Gamma_{11} \eta_1^{\kappa_1} \ge 0
  \ee
  as well as
  \be{16.7}
	\frac{\Gamma_{11}}{\Gamma_{10}} \cdot \eta_1^{\kappa_2-1} \le \frac{1}{2}.
  \ee
  We thereupon fix $\eta=\eta(p,r,a,h,f)>0$ in such a way that
  \be{16.8}
	2^p \eta + 2^p a^p |\Om|^\frac{r-p}{r} \eta^\frac{p}{r} + b_1 \eta^\frac{p}{r} 
	+ 2^\frac{p}{2} b_2 \Gamma_3^\frac{p}{2} |\Om|^\frac{p(r-2)}{2r} \eta^\frac{p}{r}
	+ 2^\frac{p}{2} b_2 \Gamma_3^\frac{p}{2} |\Om|^\frac{p-2}{2} \eta 
	\le \frac{\eta_1}{2},
  \ee
  and henceforth assume $u_0$ and $u_{0t}$ to comply with (\ref{init}) and to satisfy (\ref{16.1}).
  We then let $\F$ be as defined in (\ref{F}), and note that $\F$ is continuous on $[0,\tm)$ and has the property that
  thanks to (\ref{v}), the H\"older inequality, (\ref{13.1}), (\ref{16.1}) and (\ref{16.8}),
  \bea{16.9}
	 \F(0)
	&\le&\io |\na u_{0t} + a\na u_0|^p
	+ b_1 \cdot \bigg\{ \io |\na u_0|^r\bigg\}^\frac{p}{r} 
	+ b_2\cdot\bigg\{ \Gamma_3 \io |\na u_0|^2 + \Gamma_3 \io |\na u_{0t}|^2 \bigg\}^\frac{p}{2} \nn\\
	&\le& 2^p \io |\na u_{0t}|^p
	+ 2^p a^p \io |\na u_0|^p
	+ b_1 \cdot\bigg\{ \io |\na u_0|^r \bigg\}^\frac{p}{r} \nn\\
	& & + 2^\frac{p}{2} b_2 \Gamma_3^\frac{p}{2} \cdot\bigg\{ \io |\na u_0|^2 \bigg\}^\frac{p}{2}
	+ 2^p b_2 \Gamma_3^\frac{p}{2}\cdot\bigg\{ \io |\na u_{0t}|^2\bigg\}^\frac{p}{2} \nn\\
	&\le& 2^p \io |\na u_{0t}|^p
	+ 2^p a^p |\Om|^\frac{r-p}{r} \cdot\bigg\{ \io |\na u_0|^r\bigg\}^\frac{p}{r}
	+ b_1\cdot\bigg\{ \io |\na u_0|^r\bigg\}^\frac{p}{r} \nn\\
	& & + 2^\frac{p}{2} b_2 \Gamma_3^\frac{p}{2} |\Om|^\frac{p(r-2)}{2r} \cdot\bigg\{ \io |\na u_0|^r\bigg\}^\frac{p}{r}
	+ 2^\frac{p}{2} \Gamma_3^\frac{p}{2} |\Om|^\frac{p-2}{2} \io |\na u_{0t}|^p \nn\\
	&\le& 2^p \eta + 2^p a^p |\Om|^\frac{r-p}{r} \eta^\frac{p}{r} + b_1 \eta^\frac{p}{r} \nn\\
	& & + 2^\frac{p}{2} b_2 \Gamma_3^\frac{p}{2} |\Om|^\frac{p(r-2)}{2r} \eta^\frac{p}{r}
	+ 2^\frac{p}{2} b_2 \Gamma_3^\frac{p}{2} |\Om|^\frac{p-2}{2} \eta \nn\\
	&\le& \frac{\eta_1}{2},
  \eea
  whence, in particular,
  \bas
	T:=\sup \Big\{ T'\in (0,\tm) \ \Big| \ \F(t) \le\eta_1 \mbox{ for all } t\in (0,T')\Big\}
  \eas
  is a well-defined element of $(0,\tm]\subset (0,\infty]$.
  To see that actually $T=\tm$, and to simultaneosuly derive (\ref{16.3}), we note that in view of (\ref{v}), (\ref{16.4}) 
  and (\ref{16.5}), the definition of $T$ especially implies that
  \bea{16.12}
	\|u\|_{L^\infty(\Om)}
	+ \|u_t\|_{L^\infty(\Om)}
	&=& \|u\|_{L^\infty(\Om)}
	+ \|v-au\|_{L^\infty(\Om)} \nn\\
	&\le& (1+a)\|u\|_{L^\infty(\Om)} 
	+ \|v\|_{L^\infty(\Om)} \nn\\
	&\le& (1+a)c_1 \|\na u\|_{L^r(\Om)}
	+ c_2\|\na v\|_{L^p(\Om)} \nn\\
	&\le& (1+a) c_1 \cdot\Big( \frac{\F(t)}{b_1}\Big)^\frac{1}{p}
	+ c_2 \F^\frac{1}{p}(t) \nn\\
	&\le& (1+a)c_1\cdot\Big(\frac{\eta_1}{b_1}\Big)^\frac{1}{p}
	+ c_2 \eta_1^\frac{1}{p} \nn\\
	&\le& \del_1
	\qquad \mbox{for all } t\in (0,T),
  \eea
  and that moreover, in line with (\ref{16.6}),
  \bas
	\Gamma_{10} - \Gamma_{11} \F^{\kappa_1}(t) \ge \Gamma_{10} - \Gamma_{11} \eta_1^{\kappa_1} \ge 0
	\qquad \mbox{for all } t\in (0,T).
  \eas
  Therefore, Lemma \ref{lem15} applies so as to warrant that
  \bas
	\F'(t) + \Gamma_{10} \F(t) \le \Gamma_{11} \F^{\kappa_2}(t)
	\qquad \mbox{for all } t\in (0,T),
  \eas
  where according to (\ref{16.7})
  \bas
	\frac{\Gamma_{11} \F^{\kappa_2}(t)}{\Gamma_{10} \F(t)}
	\le \frac{\Gamma_{11}}{\Gamma_{10}} \cdot \eta_1^{\kappa_2-1}
	\le \frac{1}{2}
	\qquad \mbox{for all } t\in (0,T),
  \eas
  because $\kappa_2>1$.
  Consequently,
  \bas
	\F'(t) + \frac{1}{2} \Gamma_{10} \F(t) \le 0
	\qquad \mbox{for all } t\in (0,T),
  \eas
  so that by (\ref{16.9}),
  \bas
	\F(t) \le \frac{\eta_1}{2} e^{-\frac{1}{2}\Gamma_{10} t}
	\qquad \mbox{for all } t\in (0,T).
  \eas
  Again by continuity of $\F$, this means that necessarily, indeed, $T=\tm$, and that, once more due to (\ref{v})
  and the H\"older inequality,
  \bas
	\bigg\{ \io |\na u|^r \bigg\}^\frac{p}{r} + \io |\na u_t|^p
	&=& \bigg\{ \io |\na u|^r \bigg\}^\frac{p}{r}
	+ \io |\na v-a\na u|^p \\
	&\le& \bigg\{ \io |\na u|^r \bigg\}^\frac{p}{r}
	+ 2^p \io |\na v|^p + 2^p a^p |\Om|^\frac{r-p}{r} \cdot \bigg\{ \io |\na u|^r\bigg\}^\frac{p}{r} \\
	&\le& \big( 1+2^p a^p |\Om|^\frac{r-p}{r}\big) \cdot \frac{\F(t)}{b_1}
	+ 2^p \F(t) \\
	&\le& \Big\{ \frac{1+2^p a^p |\Om|^\frac{r-p}{r}}{b_1} + 2^p\Big\} \cdot \frac{\eta_1}{2} e^{-\frac{1}{2}\Gamma_{10} t}
	\qquad \mbox{for all } t\in (0,\tm).
  \eas
  With $\beta\equiv \beta(p,r,a,h,f):=\frac{1}{2} \Gamma_{10}$ and some $\Gamma_{12}=\Gamma_{12}(p,r,a,h,f)$ independent of the
  particular choice of $u_0$ and $u_{0t}$ this establishes (\ref{16.3}).
\qed
In light of the refined extensibility criterion asserted by Lemma \ref{lem17}, we have thus actually achieved our goal already:\abs
\proofc of Theorem \ref{theo18}. \quad
  With $\eta=\eta(p,r,a,h,f)$ taken from Lemma \ref{lem16}, assuming (\ref{init}) and (\ref{18.1}) we infer from said lemma
  that the classical solution $u$ of (\ref{0}) from Lemma \ref{lem_loc} satisfies (\ref{16.2}) and (\ref{16.3}).
  According to Lemma \ref{lem11}, (\ref{16.2}) entails that $\inf_{\Om\times (0,\tm)} h(u,u_t)>0$, and since $r\ge p$,
  from (\ref{16.3}) we particularly infer that both $u$ and $u_t$ lie in $L^\infty((0,\tm);W^{1,p}(\Om))$.
  In view of the inequalities in (\ref{18.1}), Lemma \ref{lem17} thus asserts that $\tm=\infty$, whereupon (\ref{18.2}) 
  results from (\ref{16.3}).
\qed

\bigskip

\section*{Declarations}
{\bf Funding.} \quad
The second author acknowledges support of the Deutsche Forschungsgemeinschaft (Project No. 444955436).\abs
{\bf Conflict of interest statement.} \quad
The authors declare that they have no conflict of interest.\abs
{\bf Data availability statement.} \quad
Data sharing is not applicable to this article as no datasets were
generated or analyzed during the current study.

\end{document}